\pdfoutput=1
\RequirePackage{ifpdf}
\ifpdf 
\documentclass[pdftex]{sigma}
\else
\documentclass{sigma}
\fi

\numberwithin{equation}{section}
\usepackage{booktabs}
\usepackage{tikz-cd}

\newcommand{\Tcc}{\mathcal{T}}
\newcommand{\Acc}{\mathcal{A}}
\newcommand{\A}{\mathcal A}
\newcommand{\Xcc}{\mathcal{X}}
\newcommand{\C}{\mathbb{C}}
\newcommand{\R}{\mathbb{R}}
\renewcommand{\H}{\mathbb{H}}
\newcommand{\Rcc}{\mathcal{R}}
\newcommand{\sgn}{\operatorname{sgn}}

\newcommand{\lvec}[1]{\overset{\leftarrow}{#1}}
\newcommand{\rvec}[1]{\overset{\rightarrow}{#1}}

\newcommand{\uwidehat}[1]{%
 \mathpalette\douwidehat{#1}%
}
\makeatletter
\newcommand{\douwidehat}[2]{%
 \sbox0{$\m@th#1\widehat{\hphantom{#2}}$}%
 \sbox2{$\m@th#1x$}
 \sbox4{$\m@th#1#2$}
 \dimen0=\ht0
 \advance\dimen0 -.8\ht2
 \dimen2=\dp4
 \rlap{%
 \raisebox{\dimexpr\dimen0-\dimen2}{%
 \scalebox{1}[-1]{\box0}%
 }%
 }%
 {#2}%
}
\makeatother

\newtheorem{Theorem}{Theorem}[section]
\newtheorem*{Theorem*}{Theorem}
\newtheorem{Corollary}[Theorem]{Corollary}
\newtheorem{Lemma}[Theorem]{Lemma}

\theoremstyle{definition}
\newtheorem{Definition}[Theorem]{Definition}

\newtheorem{Example}[Theorem]{Example}
\newtheorem{Remark}[Theorem]{Remark}
\newtheorem{Conjecture}[Theorem]{Conjecture}

\begin{document}

\allowdisplaybreaks

\newcommand{\arXivNumber}{2410.22620}

\renewcommand{\PaperNumber}{112}

\FirstPageHeading{}

\ShortArticleName{Geometric Leaves of Symplectic Groupoids}

\ArticleName{Geometric Leaves of Symplectic Groupoids}

\Author{Erik BRODSKY~$^{\rm a}$, Leonid CHEKHOV~$^{\rm b}$, Pranjal DANGWAL~$^{\rm c}$, Skyler HAMLIN~$^{\rm d}$, Xu LIAN~$^{\rm e}$, Michael SHAPIRO~$^{\rm b}$, Samuel SOTTILE~$^{\rm d}$ and Jenny ZHAN~$^{\rm f}$}

\AuthorNameForHeading{E.~Brodsky et al.}

\Address{$^{\rm a)}$~Department of Mathematics, Cornell University, 310 Malott Hall, Ithaca, NY 14853,
USA}
\EmailD{\mail{eab384@cornell.edu}}

\Address{$^{\rm b)}$~Department of Mathematics, Michigan State University, East Lansing, MI 48823, USA}
\EmailD{\mail{chekhov@msu.edu}, \mail{mshapiro@msu.edu}}

\Address{$^{\rm c)}$~360 Vincent Hall, 206 Church Street SE, Minneapolis, MN 55455, USA}
\EmailD{\mail{pdangwal@umn.edu}}

\Address{$^{\rm d)}$~Department of Mathematics, Stanford University, Building 380, Stanford, CA 94305, USA}
\EmailD{\mail{hamlin@stanford.edu}, \mail{sottile@stanford.edu}}

\Address{$^{\rm e)}$~School of Mathematical Sciences, Nankai University, No 94. Weijin Road,\\
\hphantom{$^{\rm e)}$}~Tianjin, 300071, P.R.~China}
\EmailD{\mail{xulian189@outlook.com}}

\Address{$^{\rm f)}$~Department of Mathematics, University of Washington,\\
\hphantom{$^{\rm f)}$}~Box 354350, C-138 Padelford, Seattle, WA 98195-4350, USA}
\EmailD{\mail{zzhan4@uw.edu}}

\ArticleDates{Received March 05, 2025, in final form October 26, 2026; Published online November 11, 2026}

\Abstract{A symplectic groupoid of triangular bilinear forms consists of pairs of matrices $(B,A)$ with $A\in \mathcal A$---the set of real unipotent upper-triangular matrices---and $B\in {\rm SL}_n({\mathbb R})$ such that \smash{$\widetilde{A}=BAB^{\mathsf T}\in\mathcal A$}. The symplectic groupoid possesses a nondegenerate symplectic structure which induces a reflection-equation Poisson algebra on $\mathcal A$. Fock and Chekhov defined a Poisson map from the Teichm\"uller space ${\mathcal T}_{g,s}$ of genus $g$ surfaces with $s$ (one or two) holes to the space of unipotent upper-triangular $n\times n$ matrices; the image of this map forms a \emph{geometric locus}. A matrix~$A$ in the geometric locus satisfies the \emph{rank condition}: \smash{$\operatorname{rank}\bigl(A+A^{\mathsf T}\bigr)\le 4$} and an additional positivity condition. We describe the coisotropic reduction of the Poisson cluster variety of $\mathcal A$ by the rank condition for $n=5$ and $6$ and show that these coisotropic reductions yield the corresponding geometric loci. Moreover, we analyze in both cases the cluster structures obtained by coisotropic reductions and recover the celebrated cluster structure of ${\mathcal T}_{2,1}$ for $n=5$ and ${\mathcal T}_{2,2}$ for $n=6$. We also discuss the positivity of the geometric cluster variables for $n=3,4,5,6$.}


\Classification{53D30; 53D20; 81R12}

\section{Introduction}

The study of the cluster structure of the \emph{symplectic groupoid} of unipotent upper-triangular $n\times n$ matrices was initiated in \cite{10.1093/imrn/rnac101,chekhov2022roots}.
This paper continues this line of research by exploring the correspondence between the cluster realization of the Teichm\"uller space $\mathcal{T}_{g,s}$ of genus $g$ surfaces with $s$ holes, where $g\ge 1$, $s\in\{1,2\}$ and the cluster structure of the \emph{symplectic groupoid} of unipotent upper-triangular $n\times n$ matrices.

Let $\mathcal{A}_n$ denote the affine space of real unipotent upper-triangular $n\times n$ matrices. The symplectic groupoid $\mathfrak{M}_n$ of the unipotent upper-triangular $n\times n$ matrices is the set of pairs
\begin{align*}
 \mathfrak{M}_n=\bigl\{(B,A)\mid B\in {\rm SL}_n,\, A\in \mathcal{A}_n \text{ such that } BAB^{\mathsf{T}}\in\mathcal{A}_n\bigr\}.
\end{align*}

The space $\mathfrak{M}_n$ possesses a natural symplectic structure \cite{Wein87}.
The pushforward of the dual Poisson bracket under the projection $(B,A)\mapsto A$ induces
the Poisson bracket on $\mathcal{A}_n$ studied in \cite{boalch2001,bondal2004} and called the \emph{Bondal Poisson bracket}.

The Teichm\"uller space $\mathcal{T}_{g,s}$ of genus $g$ Riemann surfaces with $s$ holes is a Poisson variety equipped with the renowned Goldman Poisson bracket. Define
\[
\operatorname{par}(n)=\begin{cases}
 1,&\text{if } n\text{ is odd,}\\
 2,&\text{otherwise.}
\end{cases}
\]
In \cite{ChF2000}, a Poisson map $\varphi\colon \mathcal{T}_{\lfloor (n-1)/2\rfloor,\operatorname{par}(n)}\to \mathcal{A}_n$ was constructed. The map $\varphi\colon \Tcc_{\lfloor (n-1)/2\rfloor,\operatorname{par}(n)}\to\operatorname{Im}(\varphi)$ is a finite covering map. Notice that $\dim (\Tcc_{\lfloor (n-1)/2\rfloor,\operatorname{par}(n)} )=3n-6$ and
$\dim(\mathcal{A}_n)=\binom{n}{2}$. The dimensions of the corresponding generic symplectic leaves are $3n-6-\operatorname{par}(n)$ and $\binom{n}{2} - \lfloor\frac{n}{2}\rfloor$. For $n=3$ the dimensions of symplectic leaves coincide and both equal $2$; for $n=4$ the dimensions are $4$, and for $n=5$ they are $8$. However, in the latter case, the dimensions of the spaces are different,
$\dim(\mathcal{T}_{2,1})=9$ with one central element while $\dim(\mathcal{A}_5)=10$ with two central elements.

We call the image $\operatorname{Im}(\varphi)$ \emph{a geometric locus} of $\Acc_n$. It is the union of symplectic leaves in~$\Acc_n$ called \emph{geometric leaves}. For $n\ge 6$, the dimension of a geometric leaf is smaller than the dimension of a generic symplectic leaf in~$\Acc_n$.
 In \cite{10.1093/imrn/rnac101}, Chekhov and Shapiro constructed a Poisson cluster algebra
 $\Xcc(\Acc_n)$ together with a finite Poisson map $\xi\colon \C[\Acc_n]\to\Xcc(\Acc_n)$, where~$\C[\Acc_n]$ is the Poisson algebra of regular functions on $\Acc_n$ equipped with the Bondal Poisson bracket and~$\Xcc(\Acc_n)$ is equipped with the cluster Poisson structure. We note that $\operatorname{Im}(\varphi)$ satisfies the so-called \emph{rank condition}.
 We use the rank condition to describe a coisotropic reduction of $\Xcc(\Acc_n)$ for $n=5$ and~$6$. We prove that the Poisson algebra obtained as the result of the reduction for $n=5,6$ inherits a cluster structure from $\Xcc(\Acc_n)$ and coincides with the corresponding renowned cluster algebras $\Xcc\bigl(\Tcc_{\lfloor (n-1)/2\rfloor,\operatorname{par}(n)}\bigr)$ whose seeds are labeled by ideal triangulations of hyperbolic surfaces of genus $\lfloor \frac{n-1}{2}\rfloor$ with $\operatorname{par}(n)$ holes (see Section~\ref{sec:shear}).

Note that, besides the rank condition, genuine geometric leaves must also satisfy the \emph{positivity condition}, which requires that all geometric cluster variables obtained from cluster variables of~$\mathcal A_n$ quivers by reduction maps be positive.
We call the subset of~$\mathcal A_n$ that satisfies the rank condition the \emph{pre-geometric locus}.
 It is foliated by symplectic leaves called \emph{pre-geometric leaves}. Although a complete analysis of the positivity condition for all possible schemes of coisotropic reduction is beyond the scope of this paper, in two appendices we classify all genuine geometric leaves obtained in the simplest cases of coisotropic reductions for $n=5$ and $6$. We show that these geometric leaves are preserved by the braid-group action of the $\mathcal A_n$ quiver, thus proving the consistency of the geometric reduction procedure.

In Section~\ref{sec:cluster}, we review the definitions of Poisson varieties and cluster structures compatible with Poisson brackets. In Section~\ref{sec:shear}, we recall the definition of the shear coordinates on the Teichm\"uller space $\mathcal{T}_{g,s}$. In Section~\ref{sec:symplecticgroupoid}, we recall the construction of the symplectic groupoid of the unipotent upper-triangular matrices and the corresponding Poisson bracket on $\mathcal{A}_n$, while, in Section~\ref{sec:logcanonicalgroupoid}, we describe the construction of log-canonical parameters for the symplectic groupoid~\cite{10.1093/imrn/rnac101}.

In Section~\ref{sec:geometricleafdimension}, we compare the dimensions of a pre-geometric leaf and $\mathcal{A}_n$ for $n\ge 3$. In Section~\ref{sec:rankcondition}, we state the necessary condition for a pre-geometric leaf (called \emph{the rank condition}).

\begin{Definition} For an $n\times n$ unipotent upper-triangular matrix $A$, we call the condition \smash{$ \operatorname{rank}\bigl(A+A^{\mathsf{T}}\bigr)\le 4$} \emph{the rank condition}.
\end{Definition}

Finally, in Section~\ref{sec:solutionRank} we provide solutions of the rank condition for $n=3,4,5$ and $6$.

\section{Cluster algebras and Poisson geometry}\label{sec:cluster}

\subsection{Preliminaries}
We start by recalling the definition of a cluster algebra and a compatible Poisson structure.

\begin{Definition}\label{def:quiver}
A \textit{quiver} is a directed graph which we shall represent by the tuple $Q = (V, E, s, t)$. $V$ and $E$ are the \textit{vertex} and \textit{edge sets}, respectively, and $s, t\colon E \to V$ are the \textit{source} and \textit{target maps}.
\end{Definition}

We will call the elements of the directed edge set $E$ \textit{arrows}. Moreover, for an arrow $\alpha \in E$, $s(\alpha) \in V$ is called the \textit{source} or \textit{starting} vertex and $t(\alpha) \in V$ the \textit{target} or \textit{terminal} vertex.
In keeping with cluster algebra convention, we shall require our quivers to have no oriented $1$-cycles (loops) or oriented $2$-cycles.
We consider only finite quivers with $|V|<\infty$ and $|E|<\infty$.

\begin{Definition}\label{def:quiver-mut} Let $Q$ be a quiver whose set of vertices $V$ is partitioned into two disjoint subsets~$V_m$ of \emph{mutable} vertices and $V_f$ of \emph{frozen} vertices. To each vertex $k\in V_m$, one can associate a~\emph{quiver mutation} $\mu_k$.
 This mutation $\mu_k$ at vertex $k$ transforms $Q$ into a new quiver $Q' = \mu_k(Q)$ via a sequence of three steps:
\begin{enumerate}\itemsep=0pt
\item[(1)] for each oriented two-arrow path $i \to k \to j$, introduce a new arrow $i \to j$;
\item[(2)] reverse the direction of all arrows incident to the vertex~$k$;
\item[(3)] consecutively remove all oriented 2-cycles created in step~(1).
\end{enumerate}
\end{Definition}

It is an easy exercise to show that for any quiver $Q$ and any vertex $k$ in $Q$, $\mu_k$ is an involution on $Q$.

\begin{Definition}\label{def:seed}
Let $Q$ be a quiver and let $V$ denote its vertex set. A \textit{seed} associated with $Q$ is a tuple $i = (Q, \{X_k\}_{k\in V})$, where $\{X_k\}_{k\in V}$ is a set of independent variables generating the ambient field $\mathcal F$ of rational functions of $X_1, X_2, \dots$.
\end{Definition}

We now present an equivalent definition of \emph{quiver mutation}
in terms of \emph{skew-symmetric} matrices. We restrict our attention here to \textit{skew-symmetric coefficient-free} cluster algebras.

\begin{Remark}
For a more general treatment of the subject, the reader is referred to \cite{fominzelevinsky2002,GSV10}. In the current paper, a more general definition of a \emph{skew-symmetrizable cluster algebra with geometric coefficients} is not used.
\end{Remark}

\begin{Definition}\label{def:matrixmutation}
Let $B = (b_{ij})$ be a skew-symmetric
$n \times n$ integer matrix.
For $k \in [1,n]$, the \textit{matrix mutation} $\mu_k$ in direction $k$ transforms $B$ into the
$n \times n$
matrix $\mu_k(B) = B' = (b_{ij}')$ with the entries defined as
\begin{equation}\label{eq:matrixmutation}
 b_{ij}' =
 \begin{cases}
 -b_{ij} &\text{if $i=k$ or $j=k$},\\
 b_{ij} + \dfrac{|b_{ik}|b_{kj} + b_{ik}|b_{kj}|}{2}& \text{otherwise}.
 \end{cases}
 \end{equation}
\end{Definition}

\begin{Remark}\label{rm:quiver-matrix}
The information in any directed graph can be recorded in a skew-symmetric integer matrix $B$ by defining $b_{ij}$ to be the number of arrows from vertex $i$ to $j$ minus the number of arrows from $j$ to $i$.
This matrix is called the \textit{signed incidence matrix} of the graph. Observe that since we do not allow loops in a quiver, the diagonal entries of the incidence matrix of a quiver are always $0$. Thus, quivers give rise to skew-symmetric matrices, and it is easy to verify that the above Definition~\ref{def:matrixmutation} of matrix mutations agrees with Definition~\ref{def:quiver-mut} of quiver mutations given earlier in this section.
\end{Remark}

\begin{Definition}
 A \textit{seed} in $\mathcal F$ is a pair $(\mathbf{x}, B)$, where $B$ is a skew-symmetric $n \times n$ integer matrix and $\mathbf{x} = (x_1, \dots, x_n)$ is a collection of \emph{cluster} variables forming a transcendence basis of $\mathcal F$.
The tuple $\mathbf{x}$ is called a \textit{cluster}, and
$n$ is the \emph{rank} of the cluster algebra.
\end{Definition}

\begin{Definition}
 Given a seed $(\mathbf{x}, B)$ as above, the \textit{seed mutation} $\mu_k$ ($k \in [1,n])$ transforms $(\mathbf{x}, B)$ into a new seed $\mu_k(\mathbf{x}, B) = (\mathbf{x}', B')$ such that
\begin{itemize}\itemsep=0pt
 \item $B' = \mu_k(B)$ (see formula~\eqref{eq:matrixmutation}),
 \item $x'_j=(\mathbf{x'})_j$ satisfies the so-called \emph{$x$-variable mutation rule}:
\[
x'_i=\mu_k(x_i) =
\begin{cases}
x_i^{-1} & \text{if $i = k$}, \\
x_i\bigl(1 + x_k^{-\sgn(b_{ik})}\bigr)^{-b_{ik}} & \text{if $i \neq k$},
\end{cases}
\]
where
\[
\sgn(z):=\begin{cases}
 \hphantom{-}1 & \text{if } z>0,\\
 \hphantom{-}0 & \text{if } z=0,\\
 -1 & \text{if } z<0.
\end{cases}
\]
 \end{itemize}
\end{Definition}

Let $x_1$ be the set of cluster variables in a seed~$t$. For the mutation~$\mu_k$, the allowed values of~$x_k$ are those satisfying $x_k\ne -1$. Note that the condition $x_k\ne -1$ is equivalent to the condition $\mu_k(x_k)\ne -1$ in the seed $\mu_k(t)$.

\subsection{Compatible Poisson brackets}

In this subsection, we recall some basic definitions of Poisson geometry and the notion of a~Poisson structure compatible with the cluster algebra structure first introduced in~\cite{GSV03}.

A \textit{Poisson algebra} is a commutative associative algebra $\mathcal A$ equipped with a \textit{Poisson bracket} $\{ \cdot, \cdot \}\colon \mathcal A \times \mathcal A \to \mathcal A$ that is a skew-symmetric bilinear map satisfying
\begin{enumerate}\itemsep=0pt
 \item[(1)] the \textit{Leibniz identity} $\{fg, h\} = f\{g,h\} + \{f, h\}g$, and
 \item[(2)] the \textit{Jacobi identity} $\{f, \{g, h\} \} + \{g, \{h, f\} \}+ \{h, \{f,g\}\} = 0$
\end{enumerate}
for functions $f$, $g$, $h$.

A \textit{Casimir element of $\mathcal{A}$} is an element $c \in \mathcal A$ such that $\{c , f \} = 0$ for any $f \in \mathcal A$. By an abuse of notation, we also call Casimir elements \emph{Casimir functions} or, simply, \emph{Casimirs}.
A smooth real manifold $M$ is said to be a \textit{Poisson manifold} if the algebra~$C^\infty(M)$ of smooth functions on~$M$ is a Poisson algebra. In this case, we say that $M$ is equipped with a \textit{Poisson structure}.

Let $\{\cdot, \cdot\}$ be a Poisson bracket on the ambient field $\mathcal F$ considered as the field of rational functions in $n$ independent variables with rational coefficients. We say that the collection of functions $\mathbf{f}=(f_1, \dots, f_{n})$ is \textit{log-canonical} with respect to the bracket $\{\cdot, \cdot\}$ if $\{f_i, f_j\} = \omega_{ij}f_if_j$, where $\omega_{ij}$ are constants for all $i,j \in [1,n]$. The \textit{coefficient matrix} of $\{\cdot, \cdot\}$ is defined to be $\Omega^\mathbf{f} = (\omega_{ij})_{i,j=1}^n$. Due to the antisymmetry of $\{\cdot, \cdot\}$, the coefficient matrix is evidently skew-symmetric.

\begin{Definition}
A cluster structure is \emph{compatible} with a Poisson bracket if any cluster is log-canonical with respect to the bracket.
\end{Definition}

\begin{Remark}
It is well known that any cluster structure has a compatible Poisson structure~\cite{GSV10} that, in the cluster $\mathbf{x} = (x_1, \dots, x_n)$, takes the form $\{x_i,x_j\}=B_{ij} x_i x_j$. Recall that $B$ is an integer skew-symmetric matrix.
\end{Remark}

One of the well-known examples of a cluster algebra equipped with a natural Poisson bracket is provided by the Teichm\"uller space equipped with the Goldman Poisson bracket \cite{fock2006moduli,FST08Acta, GSV05}.

\section[Description of cluster coordinates of the Teichm\"uller space T\_\{g,s\}, s>=1]{Description of cluster coordinates \\ of the Teichm\"uller space $\boldsymbol{\mathcal{T}_{g,s}}$, $\boldsymbol{s\ge 1}$} \label{sec:shear}

The topology of an oriented,
complete,
two-dimensional surface is determined by its genus $g$ and the number of boundary components, $s$. Such a surface will be denoted $\Sigma_{g,s}$.
We assume that~$g$,~$s$ satisfy the \emph{hyperbolicity condition}
$2g-2+s>0$.

The topological surface $\Sigma_{g,s}$ with $g$, $s$ satisfying the hyperbolicity condition can be equipped with a hyperbolic metric, i.e., a metric with constant curvature $-1$; the surface with a hyperbolic metric is called a \emph{hyperbolic surface}.
Any hyperbolic surface is obtained as a quotient of the upper half plane $\H = \{x+{\rm i}y \in \C \mid x,y \in \R,\, y > 0 \}$ by a discrete subgroup
$F$ of ${\rm PSL}(2,\R)$ called a Fuchsian group. Note that
$\H$ is equipped with the standard hyperbolic metric
\[
{\rm d}s^2 = \frac{{\rm d}x^2+{\rm d}y^2}{y^2}.
\] An element $g\in {\rm PSL}(2,\R)$ represented by a $2\times 2$ real matrix
$g=\pm \left(\begin{smallmatrix} a & b\\
c & d
\end{smallmatrix}\right)$
 acts naturally on $\H$ via M\"obius transformations,
$z \mapsto \frac{az + b}{cz+d}$.
This action is by isometries.
On $\H$, the geodesics are given by semicircles centered on the real line and vertical lines. An ideal triangle in $\H$ is a triangle in the
open
upper half plane $\{(x,y)\mid y> 0\}$ with geodesic sides and all three vertices on the real line, or two on the real line and one at infinity. All ideal triangles in $\H$ are isometric to each other. An \emph{ideal triangle} of a hyperbolic surface is a triangle with geodesic sides isometric to one (and, hence, any) ideal triangle in $\H$.

For $2g-2 + s > 0$, each hyperbolic surface of genus $g$ with $s$ holes can be represented as a disjoint union of $s$
topologically closed ``outer'' infinite hyperbolic domains (``holes''), each of which is isometric to a half of a one-sheeted hyperboloid whose geodesic boundary is a unique closed geodesic homeomorphic to a loop around a hole (a hole ``boundary'', or perimeter), and the remaining (topologically open but bounded) part, which we call the \emph{reduced surface}. In what follows, a \emph{hyperbolic surface} will always refer to this reduced surface.
(The simplest reduced surface, $\Sigma_{0,3}$, is depicted in Figure~\ref{fi:partition}.)

\begin{figure}[ht]
 \centering
 \includegraphics[width=70mm]{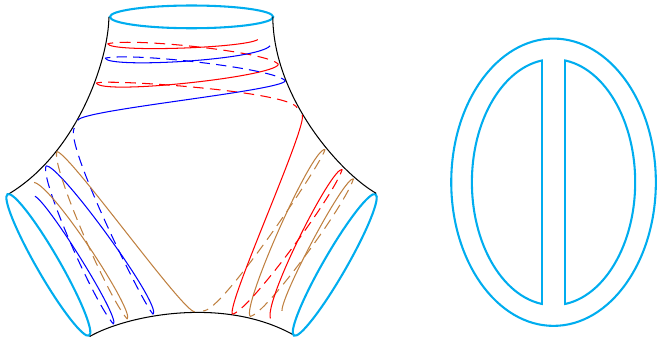}
\caption{Splitting $\Sigma_{0,3}$ into two ideal triangles; three lines of the partition, or sides of the triangles, are indicated by three different colors. Each line has an infinite geodesic length, and each of its ends winds indefinitely toward the perimeter of a hole. The corresponding fat graph has three edges dual to the partition lines.}\label{fi:partition}
\end{figure}

A fundamental domain is a subset of $\H$ that is a connected union of $4(g-1)+2s$ ideal triangles, which, upon the action of $F$ on $\H$, transforms into an infinite ideal-triangle tessellation~$\mathbb T$ of $\H-\bigl\{ \bigcup^\infty C_{P_i} \bigr\}$, where $C_{P_i}$ are disjoint closed semicircular domains bounded by all copies of boundary geodesics. Then $\Sigma_{g,s}=\bigl( \H-\bigl\{ \bigcup^\infty C_{P_i} \bigr\} \bigr)\bigl\backslash F$, and we call an (ideal) triangulation $T$ of~$\Sigma_{g,s}$ the corresponding factor of the tessellation, $T=\mathbb T\backslash F$.

Therefore, an \emph{ideal triangulation} of a hyperbolic surface is a partition of the reduced surface by geodesics that satisfies the following conditions:
\begin{itemize}\itemsep=0pt
\item The surface is divided into regions, each of which is the interior of an ideal triangle.
\item Each geodesic in the partition forms an edge of an embedded ideal triangle.
\item The vertices of the ideal triangles are located at ideal points, which are points at infinity in hyperbolic geometry
(the absolute).
\end{itemize}

In an ideal triangulation of a hyperbolic surface, each edge is a common side of two triangles that together form an ideal quadrilateral.
A parameter, \emph{the $($exponential$)$ shear coordinate}, can be assigned to each edge of an ideal triangulation by lifting to~$\H$.
The exponential shear coordinate is the cross-ratio of the four vertices of the lift of the ideal quadrilateral
and is always a~positive real number.

An ideal triangulation~$T$ of $\Sigma_{g,s}$ is dual to the \emph{ribbon} (or \emph{fat}) graph $\Gamma$ whose vertices correspond to ideal triangles of $T$ and edges of $\Gamma$ are dual to the edges of~$T$, that is, each edge of $\Gamma$ intersects transversely exactly one edge of $T$ at exactly one point. The ribbon graph is homotopy equivalent to the surface,
in particular, $\pi_1(\Sigma_{g,s},p)\simeq \pi_1(\Gamma,p)$ for any point $p\in\Gamma$. Any
nontrivial element $\gamma\in \pi_1(\Sigma_{g,s},p)$ has a unique representative as a closed path (sequence of edges of $\Gamma$ with the same starting and end point)
without backtracking.
 At the vertices of the ribbon graph, the incident half-edges have a cyclic order that is compatible with the orientation of the surface.

\begin{Example} The simplest example is $\Sigma_{0,3}$---the sphere with three holes also known as a pair of pants. The corresponding partition into ideal triangles and the fat graph dual to this partition are depicted in Figure~\ref{fi:partition}.
\end{Example}

On a hyperbolic surface $\Sigma$, a closed geodesic $\gamma$ starting and ending at $p \in \Sigma$ can be lifted to a path $\tilde{\gamma}$ in the universal cover~$\H$. The deck transformation of $\H$ given by $[\gamma] \in \pi_1(\Sigma)$ is represented by an element $M_\gamma\in {\rm PSL}(2,\R)$ of the Fuchsian group of the surface. Following~\cite{fock1993}, we describe $M_\gamma$ as a product of matrices:
\begin{gather*}
L = \begin{bmatrix} \hphantom{-}0 & \hphantom{-}1\\ -1 & -1\end{bmatrix},\qquad
 R = \begin{bmatrix} \hphantom{-}1 & 1\\ -1 & 0\end{bmatrix},\qquad
 X(x_t) = \begin{bmatrix} 0 & -\dfrac{1}{\sqrt{x_t}} \\ \sqrt{x_t} & 0\end{bmatrix},\nonumber\\
L_s:=LX(s)=\begin{pmatrix} s^{-1/2}&0\\ -s^{-1/2}&s^{1/2} \end{pmatrix},\qquad R_s:=RX(s)=\begin{pmatrix} s^{-1/2}&-s^{1/2}\\0 &s^{1/2} \end{pmatrix},
\end{gather*}
where $t$ is an edge of a ribbon graph describing the hyperbolic surface and $x_t>0$ is the (exponential) shear parameter assigned to $t$. The path $\gamma$ travels along the edges of the ribbon graph, turning right and left, which leads to the following product formula for the deck transformation~$M_\gamma$ with the composition law from right to left.
Then, the \emph{geodesic function}
$G_\gamma := |{\rm tr}M_\gamma |$ is given by the well-known formula $G_\gamma = 2\cosh \bigl(\frac{l_\gamma}{2}\bigr)$~\cite{ChF2000} in terms of the geodesic length $l_\gamma$.

\begin{figure}[t]
 \centering
 \includegraphics[width=34mm]{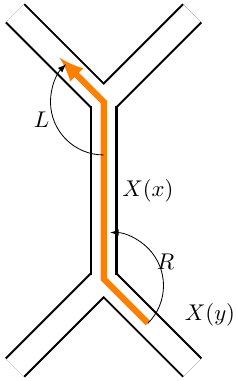}\vspace{-1mm}
\caption{Factorization of the deck transformation operator. The deck transformation along path $\gamma$ is expressed as a product of the matrix sequence
$M_\gamma = \cdots L\cdot X(x)\cdot R\cdot X(y)\cdots =\cdots L_x\cdot R_y \cdots$.}\label{fig:hol}\vspace{-1mm}
\end{figure}

\begin{Definition}
 The Teichm\"uller space $\mathcal{T}_{g,s}$ is the space of all hyperbolic surfaces of genus $g$ with $s$ holes (or, equivalently, complex structures on a genus $g$ surface with $s$ holes with a choice of basis of $H_1(\Sigma_{g,s},\mathbb{Z})$).
\end{Definition}

\section[Description of the space of objects A\_n of the symplectic groupoid M\_n]{Description of the space of objects $\boldsymbol{\mathcal{A}_n}$\\ of the symplectic groupoid $\boldsymbol{\mathfrak{M}_n}$}\label{sec:symplecticgroupoid}

\begin{Definition}[see~\cite{bondal2004}]
 Bondal's \emph{symplectic groupoid $\mathfrak{M}_n$ of unipotent upper triangular matrices} (or, \emph{symplectic groupoid}, for brevity) is the groupoid with object set $\mathcal{A}_n$, the set of $n\times n$ unipotent real upper triangular matrices, and the set of morphisms from $A \in \mathcal{A}_n$ to \smash{$\widetilde A \in \mathcal{A}_n$}, i.e., the set of matrices $B \in {\rm SL}(n, \R)$ such that
 $\widetilde A=BAB^{\mathsf{T}}$.
\end{Definition}

For a more general notion of symplectic groupoid, see \cite{Wein87}.
Note that $\mathfrak{M}_n$ is a submanifold of ${\rm SL}(n,\R) \times \mathcal{A}_n$ given by
$\mathfrak{M}_n = \bigl\{(B,A) \mid B \in {\rm SL}(n,\R),\, A \in \mathcal{A}_n,\, BAB^{\mathsf{T}} \in \mathcal{A}_n\bigr\}$.

 The \emph{set of compatible pairs} is the fiber product $\mathfrak{M}_n\times_{\mathcal{A}_n}\mathfrak{M}_n$:
\begin{gather*}
\mathfrak{M}_n^{(2)} = \bigl\{\bigl( \bigl(C,\widetilde A\bigr), (B, A) \bigr) \mid C,B \in {\rm SL}(n,\R), \, A,\widetilde A \in \mathcal{A}_n, \\
\hphantom{\mathfrak{M}_n^{(2)} = \bigl\{\bigl( \bigl(C,\widetilde A\bigr), (B, A) \bigr) \mid}{}
 \ \widetilde A = BAB^{\mathsf{T}},\, C\widetilde A C^{\mathsf{T}} \in \mathcal{A}_n\bigr\} \subset \mathfrak{M}_n \times \mathfrak{M}_n.
\end{gather*}
There are two natural projections $p_1,p_2\colon \mathfrak{M}_n^{(2)}\to\mathfrak{M}_n$, projecting on the first and second coordinates respectively, in addition to a multiplication map
$m\colon \bigl(\bigl(C, BAB^{\mathsf{T}}\bigr),(B,A)\bigr) \mapsto (CB,A)$.
There is a natural symplectic form $\omega$ on $\mathfrak{M}_n$ such that
$m^*\omega = p_1^*\omega + p_2^*\omega$~\cite{bondal2004}. The symplectic form
$\omega$ defines a dual Poisson bracket on $\mathfrak{M}_n$.
The pushforward of this Poisson bracket under the natural projection $\mathfrak{M}_n\to\mathcal{A}_n$ induces a Poisson bracket $\mathcal{P}_{{\rm sym}}$ on $\mathcal{A}_n$, studied in~\cite{bondal2004}.

In particular, for any matrix $A\in\mathcal{A}_n$, consider the polynomial $\chi_A(\lambda):=\det\bigl(A+\lambda A^{\mathsf{T}}\bigr)$. Its coefficients are polynomial functions of the entries $A_{ij}$ of~$A$. They
generate the algebra of Casimir functions with respect to the Poisson bracket on~$\mathcal{A}_n$.
It is easy to see that $\chi_A(\lambda)=\chi_{\widetilde A}(\lambda)$.
\begin{Remark}
 Note that the polynomial $\chi_A(\lambda)$ is reciprocal for any $A$ and, hence, the coefficients contain only $\lfloor \frac{n}{2}\rfloor$ algebraically independent functions of $A_{ij}$.
\end{Remark}

\section[Log-canonical coordinates for A\_n]{Log-canonical coordinates for $\boldsymbol{\mathcal{A}_n}$}\label{sec:logcanonicalgroupoid}

In this section we recall the construction from \cite{10.1093/imrn/rnac101} of log-canonical coordinates compatible with the Poisson bracket on $\mathcal{A}_n$ . The Fock--Goncharov parameters $Z_{ijk}$ for $n\times n$ upper-triangular matrices are log-canonical coordinates for the standard trigonometric Poisson-Lie bracket on~${\rm SL}_n$. Parameters $Z_{ijk}$ are organized in the triangular lattice. The corresponding log-canonical Poisson bracket $\mathcal{P}_{\rm tr}$ is described by the quiver on Figure~\ref{fig:FG}. As explained in~\cite{10.1093/imrn/rnac101}, $Z_{ijk}$ parametrize, in particular, the Borel subgroup $B_n$ of the upper triangular matrices
which is a Poisson subgroup of~${\rm SL}_n$ and the log-canonical Poisson structure
 $\mathcal{P}_{\rm tr}$ coincides with the
 Poisson-Lie structure on~$B_n$.\looseness=-1

\begin{figure}[t]
\centering
\includegraphics[width=85mm]{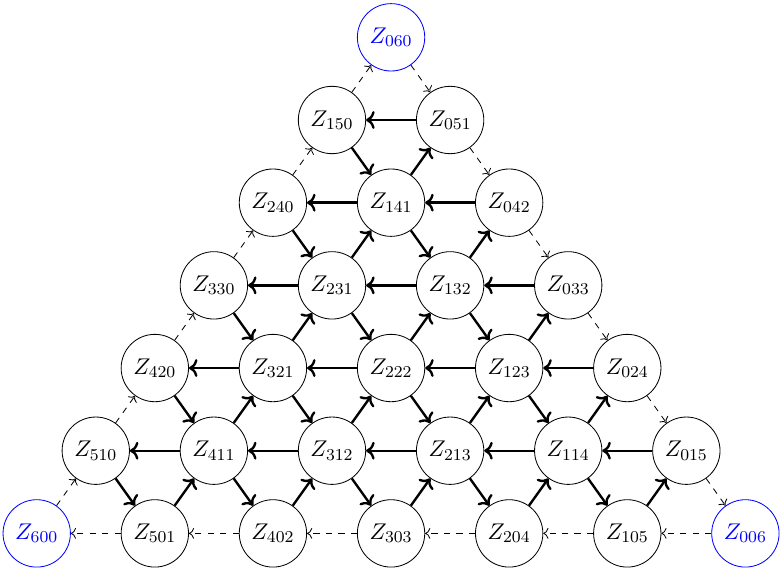}
\caption{The Fock--Goncharov parameters organized in the triangular lattice. The quiver with vertices labeled by the Fock--Goncharov parameters $Z_{\alpha,\beta,\gamma}$ encodes the Poisson bracket between $Z$'s: $\{Z_{\alpha,\beta,\gamma},Z_{\alpha',\beta',\gamma'}\}=\varkappa\cdot Z_{\alpha,\beta,\gamma}\cdot Z_{\alpha',\beta',\gamma'}$. Here $\varkappa$ denotes the signed weight of the arrow between $(\alpha,\beta,\gamma)$ and $(\alpha',\beta',\gamma')$.
The coefficient $\varkappa=1$ if there is a solid arrow from $(\alpha,\beta,\gamma)$ to $(\alpha',\beta',\gamma')$, $\varkappa=-1$ if there is a solid arrow in the opposite direction;
$\varkappa=\frac{1}{2}$ if the arrow from $(\alpha,\beta,\gamma)$ to $(\alpha',\beta',\gamma')$ is dashed,
$\varkappa=-\frac{1}{2}$ if the arrow in the opposite direction is dashed;
and $\varkappa=0$ if $(\alpha,\beta,\gamma)$ and $(\alpha',\beta',\gamma')$ are disconnected.}
\label{fig:FG}
\end{figure}

As the next step in~\cite{10.1093/imrn/rnac101}, we use $Z_{ijk}$ to construct log-canonical coordinates for the Poisson bracket on $\mathcal{A}_n$.
Parameters $Z_{ijk}$ describe three (non-independent) triangular matrices ${\mathcal M}_1$, ${\mathcal M}_2$, ${\mathcal M}_3$.
Matrices $\mathcal{M}_i$ are boundary measurement matrices in the sense of A.~Postnikov~\cite{postnikov2006}. Below, we explain the construction in the case $n=6$.
We draw the following oriented graph $\Upsilon$ (see the black directed graph in Figure~\ref{fi:plab}) dual to the original quiver (in blue). Let ${\mathfrak B}$ denote the set of triples of barycentric indices ${\mathfrak B}=\{(i_1,i_2,i_3) \mid 0\le i_1,i_2\,i_3\le 6 \text{ and } i_1+i_2+i_3=6\}$.
For each oriented path $p$ in $\Upsilon$ starting from one of the vertices $\{1,2,3,4,5,6\}$ on the right side of the triangle and terminating at one of the vertices $\{1',2',3',4',5',6'\}$ on the left side of the triangle, we denote by ${\mathfrak B}_p$ the subset of $\mathfrak{B}$ corresponding to vertices situated above the path $p$ and we assign to the path its weight $w(p)=\prod_{\alpha\in {\mathfrak B}_p} Z_\alpha$.

Let ${\mathfrak P}_{ij}$ be the set of all oriented paths from $j$ to $i'$. We define the entry $({\mathcal M}_1)_{ij}$ of the matrix~${\mathcal M}_1$ to be $({\mathcal M}_1)_{ij}=\sum_{p\in {\mathfrak P}_{ij}} w(p)$. Clearly, ${\mathcal M}_1$ is lower-triangular because there is no path in~$\Upsilon$ from~$j$ to $i'$ for $i<j$.
To construct ${\mathcal M}_2$ we consider oriented paths in $\Upsilon$ from the right side
$\{1,2,3,4,5,6\}$ to the bottom
$\{1'',2'',3'',4'',5'',6''\}$. Denote the set of such paths connecting~$j$ to~$i''$ by~${\mathfrak P}^1_{ij}$. Any such path $p$ divides the triangle into two parts: the left part~${\mathfrak B}^l_p$ containing $Z_{006}$ and $Z_{600}$ and the right part ${\mathfrak B}^r_p$ containing $Z_{060}$. We recall that  Fock--Goncharov parameters are obtained by restriction  $Z_{00n}=Z_{0n0}=Z_{n00}=1$ and set the weight $w(p)=\prod_{\alpha\in{\mathfrak B}^l_p} Z_\alpha$, and $\left({\mathcal M}_2\right)_{ij}=\sum_{p\in {\mathfrak P}^1_{ij}} w(p)$. ${\mathcal M}_2$ is upper-triangular.

Finally, to construct ${\mathcal M}_3$, we obtain a new graph $\Upsilon'$ from $\Upsilon$ by rotating the graph $\Upsilon$ shown in black in Figure~\ref{fi:plab} by $120^\circ$ counterclockwise. It is important to note that the labels of all variables and all boundary vertices do not rotate; only the orientations of the arrows in~$\Upsilon$ change. Consider oriented paths in $\Upsilon'$ from the left side
$\{1',2',3',4',5',6'\}$ to the bottom
$\{1'',2'',3'',4'',5'',6''\}$. Denote the set of such paths connecting $j'$ to $i''$ by ${\mathfrak P}^2_{ij}$. Any such path $p$ divides the triangle in two parts: the left part \smash{${\widetilde{\mathfrak B}}^l_p$} containing $Z_{600}$ and the right \smash{${\widetilde{\mathfrak B}}^r_p$} containing~$Z_{060}$ and~$Z_{006}$. After restriction  $Z_{00n}=Z_{0n0}=Z_{n00}=1$, we set the weight $w(p)=\prod_{\alpha\in{\widetilde{\mathfrak B}}^l_p} Z_\alpha$, and $({\mathcal M}_3)_{ij}=\sum_{p\in {\mathfrak P}^2_{ij}} w(p)$. ${\mathcal M}_3$ is upper-triangular.

\begin{figure}[t]
\centering
\includegraphics[width=65mm]{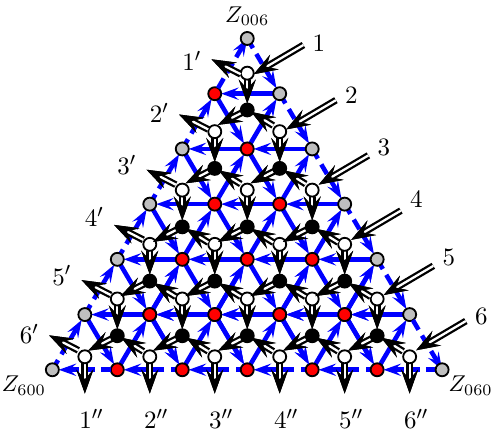}

\caption{The plabic graph $\Gamma$ dual to the quiver of Fock--Goncharov parameters.}\label{fi:plab}
\end{figure}

Next, define the normalized matrices
\begin{equation}\label{eq:DandS}
M_i=S\cdot \mathcal{M}_i\cdot D_i^{-1},
\end{equation}
where
\[
D_i=(\det({\mathcal M}_i))^{1/n}
\qquad \text{and}\qquad
S=\begin{pmatrix}
 0 & \dots & 0 & 1 \\
 0 & \dots & -1 & 0 \\
 & \dots & \dots & \\
 0 & (-1)^{n-1} & \dots & 0 \\
 (-1)^n & 0 & \dots & 0
\end{pmatrix}.
\]

In~\cite{10.1093/imrn/rnac101}, it was proposed to set $\widehat{A}=M_1^{\mathsf{T}} M_2$. It is easy to observe that $\widehat{A}$ is upper triangular. Upon normalization discussed below, matrix elements~$a_{i,j}$
 of $\widehat A$ satisfy Poisson relations similar to those satisfied by geodesic functions defined in Section~\ref{sec:geometricleafdimension}. We call $a_{i,j}$
 \emph{$($generalized$)$ geodesic functions}.

The upper-triangular matrix $\widehat{A}$ is constructed from the pair $(\mathcal{M}_1,\mathcal{M}_2)$ thereby defining a~Poisson map from the collection of Fock--Goncharov parameters equipped with the Poisson structure~$\mathcal{P}_{\rm tr}$ to the upper-triangular matrices equipped with the symplectic groupoid Poisson bracket~$\mathcal{P}_{\text{sym}}$. The expression of~$a_{ij}$ in terms of Fock--Goncharov parameters depends neither separately on $Z_{k,0,n-k}$ nor on $Z_{k,n-k,0}$ but only on the product $Z_{k,0,n-k}Z_{n-k,k,0}$ for all $k=1,\dots,n-1$.
Therefore, in the expression of matrix entries $a_{ij}$ parameters $Z_{k,0,n-k}$ and $Z_{n-k,k,0}$ are replaced by their products $Z_{k,0,n-k} Z_{n-k,k,0}$.
Replacing a pair of parameters $Z_{k,0,n-k}$ and $Z_{n-k,k,0}$ by their product leads to a new quiver, which is obtained from $\Gamma$ by the process of \emph{amalgamation}. During the amalgamation, the vertices carrying $Z_{k,0,n-k}$ and $Z_{n-k,k,0}$ are merged into a single vertex carrying $Z_{k,0,n-k} Z_{n-k,k,0}$. The set of arrows incident to the new vertex is the union of arrows incident to either of the two original vertices. The new quiver obtained by amalgamation of all the pairs $Z_{k,0,n-k}$ and $Z_{n-k,k,0}$ for $k=1,\dots,n-1$ has $n-1$ fewer vertices.
The amalgamated quiver has $n-1$ remaining frozen variables $Z_{0,k,n-k}$.
The log-canonical Poisson bracket on $Z_{ijk}$ determined by the original quiver induces the
log-canonical Poisson bracket for amalgamated parameters determined by the amalgamated quiver.
Amalgamation results in the appearance of exactly $n-1$ new independent Casimir functions \cite{10.1093/imrn/rnac101} $\mathfrak{C}_k=\prod_{i=0}^{n-k} Z_{n-k-i,k,i}\cdot\prod_{j=0}^k Z_{j,k-j,n-k}$ for all $k=1,\dots,n-1$. Note that only the first powers of $Z_{ij\ell}$, $i\ne 0$ and exactly one square $Z_{0,k,n-k}^2$, $k= 1,\dots,n-1$ enter $\mathfrak{C}_k$.

After choosing the symplectic leaf $\mathfrak{L}=\{\mathfrak{C}_k=1\}$ for all $k=1,\dots, n-1$, any collection of parameters in $\mathfrak{L}$ determines the upper-triangular matrix $\hat{A}$ with \emph{units on the main diagonal}. The \emph{normalizing condition} $\mathfrak{C}_k=1$ implies
\begin{equation}\label{eq:sqroot}
Z_{0,k,n-k}=\dfrac{1}{\sqrt{\prod_{i=0}^{n-k-1} Z_{n-k-i,k,i}\cdot\prod_{j=1}^k Z_{j,k-j,n-k}}}
\end{equation}
 for $k=1,\dots,n-1$.

Hence, on the symplectic leaf $\mathfrak{L}$, the frozen parameter $Z_{0,k,n-k}$ can be expressed in terms of the other mutable parameters and is dropped from the set of coordinates.

\begin{Remark}
Formula~\eqref{eq:sqroot} explains the appearance of half-integer powers in the expressions for $a_{i,j}$ if $\hat{A}$ satisfies the normalizing condition; see examples in Section~\ref{sec:solutionRank}.
\end{Remark}

\begin{figure}[t]
\centering
\includegraphics[width=75mm]{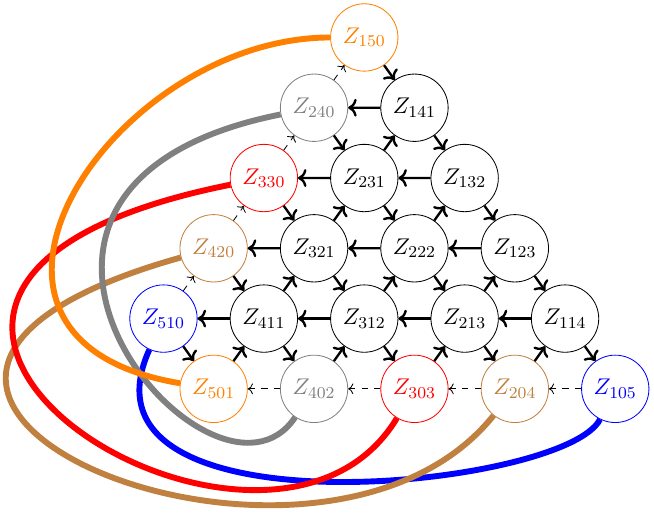}

\caption{The quiver shows the collection of amalgamated Fock--Goncharov parameters. The amalgamated pairs of variables are shown in the same color and connected by an additional thick curve of the same color.
This amalgamated quiver is the one depicted in Figure~\ref{fig:clusterA6},
where parameters $a_t,b_t$ and $c_t$ coincide either with $Z_{ijk}$ or with amalgamated
$Z_{k,0,n-k} Z_{n-k,k,0}$.}\label{fig:amalg}
\end{figure}

\begin{Remark}Note that after amalgamation and imposing the normalizing condition, the quiver in Figure~\ref{fig:FG} becomes the quiver in Figure~\ref{fig:amalg} for the remaining Fock--Goncharov parameters, which is equivalent to the quiver in Figure~\ref{fig:clusterA6}.
\end{Remark}

\section[Geometric leaf and dimension count: The cases n=3 and n=4]{Geometric leaf and dimension count:\\ The cases $\boldsymbol{n=3}$ and $\boldsymbol{n=4}$}\label{sec:geometricleafdimension}

We begin by recalling the known parametrization of the Teichm\"uller space $\mathcal{T}_{g,s}$ of two-dimen\-sional hyperbolic surfaces of genus $g$ with one or two holes.
Following \cite{ChF2000}, for a positive integer $n\ge 3$ consider the following ribbon graph $\Gamma_n$ consisting of two horizontal chains of $n-3$ equal intervals and
$n$ slanted edges connecting vertices of the first chain with vertices of the second chain (see Figure~\ref{fign5v1} below for $n=5$).

\begin{figure}[ht]
\centering
\includegraphics[width=45mm]{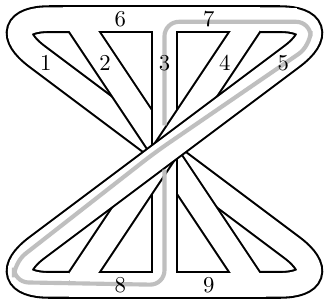}

\caption{Ribbon graph $\Gamma_5$ describes the surface of genus~2 with one hole. The gray curve represents the loop
$\gamma_{35}$. The matrix of the monodromy operator along the loop
$\gamma_{35}$ is $M_{\gamma_{35}}=L_{x_5} R_{x_8} L_{x_3} R_{x_7}$.}\label{fign5v1}
\end{figure}

We label the slanted edges $1$ through $n$ in left-to-right order of their upper endpoints.
Recall, $\operatorname{par}(n)= 1$, for odd $n$ and $ 2$, for even~$n$.
Note that the ribbon graph $\Gamma_n$ is homeomorphic to a~surface $\Sigma_{g,s}$ of genus $g=\lfloor\frac{n-1}{2}\rfloor$
 with $s=\operatorname{par}(n)$ holes.

The cardinality of the set of edges $E(\Gamma_n)$ of the ribbon graph $\Gamma_n$ (Figure~\ref{fign5v1}) is
$3n-6=\dim \mathcal{T}_{g,s}$. The collection of positive parameters $x_i=x(e_i)>0$ for all edges $e_i\in E(\Gamma_n)$, parametrizes
the Teichm\"uller space $\mathcal{T}_{g,s}$ that carries the renowned Goldman Poisson bracket.
Parameters $x_i$ are \emph{log-canonical coordinates} for the Goldman bracket
\begin{equation}\label{eq:Gldmn}
\{x_i,x_j\}={\rm adj}(e_i,e_j) x_i x_j,
\end{equation}
where ${\rm adj}(e_i,e_j)$ is the adjacency index of edges $e_i$ and $e_j$.
Namely, for $e\in E(\Gamma_n)$ we denote by~$\partial e$ the subset $\{P,Q\}$ of two (possibly, coincident) vertices of $e$; i.e., elements of the set $V(\Gamma_n)$ of vertices of $\Gamma_n$ such that $e$ connects $P$ and $Q$.
\[{\rm adj}(e_i,e_j)=\sum_{{\bf v}\in \partial e_i\cap \partial e_j} {\rm adj}_{\bf v}(e_i,e_j),\]
and
the local adjacency index
\[
{\rm adj}_{\bf v}(e_i,e_j)=
\begin{cases}
 +1&\text{if at vertex $\bf v$ the half-edge $e_j$ immediately follows $e_i$} \\
 &\text{in the anti-clockwise direction},\\
 -1&\text{if at vertex $\bf v$ the half-edge $e_j$ immediately follows $e_i$} \\
&\text{in the clockwise direction}\\
\end{cases}
\] (see Figure~\ref{fig:adj}).

\begin{figure}[t]
\centering
\includegraphics[width=20mm]{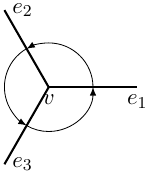}

\caption{Adjacency indices: ${\rm adj}_{\bf v}(e_1,e_2)={\rm adj}_{\bf v}(e_2,e_3)={\rm adj}_{\bf v}(e_3,e_1)=1$;
${\rm adj}_{\bf v}(e_1,e_3)={\rm adj}_{\bf v}(e_3,e_2)={\rm adj}_{\bf v}(e_2,e_1)=-1$.}\label{fig:adj}
\end{figure}

Note that adjacency index ${\rm adj}(e_i,e_j)$ is well-defined for any pair of edges of a ribbon graph $R$ having only 3-valent vertices. We use it to associate with $R$ a quiver $\mathfrak{Q}_R$ with $|E(R)|$ vertices, called the \emph{geometric quiver}. The vertices are in one-to-one correspondence with edges of $R$. The number of arrows from vertex $e_i$ to vertex $e_j$ equals ${\rm adj}(e_i,e_j)$ if ${\rm adj}(e_i,e_j)>0$.

A path in $\Gamma_n$ is a sequence $P_1,P_2,\dots,P_m$ of vertices such that each pair of consecutive vertices~$P_i$,~$P_{i+1}$ is joined by an edge of~$\Gamma_n$. A~path is a loop if $P_m=P_1$. We call a path (or loop) \emph{simple} if there are no repeated edges in the path.

We denote by $\gamma_{ij}$ the simple loop whose starting point is the top vertex of the slanted edge~$i$, followed by the bottom vertex of slanted $i$, followed by the simple path of bottom vertices from the bottom vertex of~$i$ to the bottom vertex of $j$, followed by the top vertex of slanted~$j$, followed by the simple path of top vertices from the top vertex of $j$ to the top vertex of~$i$ (see Figure~\ref{fign5v1}). Note that $\gamma_{ij}$ is completely determined by the choice of a~pair of indices $1\le i,j\le n$. Note that~$\gamma_{ii}$ is the identity in $\pi_1(\Gamma_n)$ as a trivial loop.

For any closed geodesic path (a loop, possibly with self-intersections) $\gamma$ in a hyperbolic surface $\Sigma^h$, we denote by $M_\gamma\in {\rm PSL}_2(\mathbb{R})$ the corresponding deck transformation-isometry of the universal cover $\mathbb{H}$.
The geodesic function is given by $G_\gamma=|{\rm tr}(M_\gamma)|$. Note that the absolute value of the trace is well-defined for an element of ${\rm PSL}_2(\mathbb{R})$.

The \emph{Goldman Poisson bracket} for geodesic functions of two loops $\alpha$ and $\beta$ intersecting once takes the following form:
\begin{equation}\label{eq:Goldman}
\{G_\alpha,G_\beta\}=\dfrac{1}{2}G_{\alpha\beta}-\dfrac{1}{2}G_{\alpha\beta^{-1}}
\end{equation} (see Figure~\ref{fig:Goldman}).

\begin{figure}[t]
\centering
\includegraphics[width=60mm]{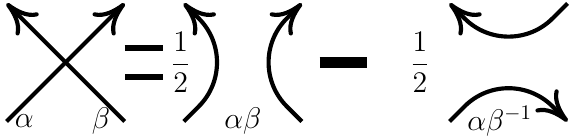}

\caption{Goldman Poisson bracket.}\label{fig:Goldman}
\end{figure}

\begin{Definition}\label{def:Gij}
$G_{ij}=|{\operatorname{tr}{M_{\gamma_{ij}}}}|$ is the geodesic function corresponding to $\gamma_{ij}$, where $M_{\gamma_{ij}}\in {\rm PSL}_2(\mathbb{R})$ is the element of the Fuchsian group determined by the loop~$\gamma_{ij}$.
\end{Definition}

\begin{Remark} As mentioned above, $G_{ij}=e^{\ell_{ij}/2}+e^{-\ell_{ij}/2}$, where $\ell_{ij}$ is the geodesic length of~$\gamma_{ij}$. From the formulas of Section~\ref{sec:shear}, it is easy to find the expression for~$G_{ij}$ in terms of parameters~$x_i$. The pushforward of the Goldman bracket~\eqref{eq:Goldman} coincides with the Bondal Poisson bracket on the unipotent triangular matrices~$\mathcal{A}_n$ induced from the symplectic groupoid.
\end{Remark}

\begin{Remark}\label{rem:centrallysymmetric}
Note that the graph $\Gamma_n$ is centrally symmetric. This central symmetry generates the group~${\mathbb Z}_2$. Any slanted edge of~$\Gamma_n$ is centrally symmetric, and any edge in the top horizontal part of~$\Gamma_n$ is centrally symmetric to the corresponding edge of the bottom horizontal part. For instance,
each of edges 1--5 in Figure~\ref{fign5v1} is centrally symmetric,
edges~6 and~9 form a centrally symmetric pair, edges~7 and~8 form the other centrally symmetric pair.
Therefore, $\mathbb{Z}_2$ acts on~$E(\Gamma_n)$ by exchanging edges in all pairs of two distinct centrally symmetric edges. This action induces $\mathbb{Z}_2$-action on the corresponding Teichm\"uller space~$\mathcal{T}_{g,s}$ and on the set of loops in $\Gamma_n$.

This $\mathbb{Z}_2$-action on the set of loops changes the orientation of any loop $\gamma_{ij}$ to the opposite one. Since the geodesic function $G_{ij}$ does not depend on the orientation of the loop $\gamma_{ij}$, $G_{ij}$ is invariant under $\mathbb{Z}_2$-action.
\end{Remark}

\begin{Definition}
A Poisson subalgebra of a Poisson algebra $\mathcal{A}$ is a vector subspace of $\mathcal{A}$ closed under multiplication and Poisson bracket. A Poisson subalgebra is generated by a subset $S$ if it is a minimal Poisson subalgebra containing $S$.
\end{Definition}

Let $\mathbf{R}_n$ be the Poisson subalgebra generated by $\{G_{i i+1}\}_{1\le i< n}$. Let
$\mathcal{O}[\mathcal{T}_{g,s}]^{\mathbb{Z}_2}\subset \mathcal{O}[\mathcal{T}_{g,s}]$ be the subring of $\mathbb{Z}_2$-invariant functions of the ring of regular functions on $\mathcal{T}_{g,s}$. Remark~\ref{rem:centrallysymmetric} implies the following statement.

\begin{Lemma}\label{lem:Z2-invariant}
$\mathbf{R}_n\subseteq \mathcal{O}[\mathcal{T}_{g,s}]^{\mathbb Z_2}$.
\end{Lemma}

The Chekhov--Fock map $\varphi\colon \mathcal{T}_{g,s}\to \mathcal{A}_n$ was defined in~\cite{ChF2000} as follows. For a point $\Sigma^h\in \mathcal{T}_{g,s}$, denote by $G_{ij}(\Sigma^h)$ the value of the geodesic function $G_{ij}$ at $\Sigma^h$. Then,
\begin{align*}
 \varphi(\Sigma^h)=
 \begin{pmatrix}1&&G_{ij}(\Sigma^h)\\&\ddots&\\0&&1\end{pmatrix}.
\end{align*}

Recall the following statement.

\begin{Theorem}[\cite{ChF2000}]
The map $\varphi\colon \mathcal{T}_{g,s}\to \mathcal{A}_n$ is a Poisson morphism that sends the Goldman Poisson bracket on the Teichm\"uller space to the Bondal Poisson bracket on the space of the unipotent upper-triangular matrices.
\end{Theorem}

It was shown in~\cite{ChF2000} that for a special choice of the geodesics $G_{i,j}$ shown in Figure~\ref{fign5v1}, applying the skein and Poisson relations (the Goldman bracket), a Poisson algebra generated by the finite set $G_{i,j}$, $1\le i<j\le n$, becomes closed. This algebra of geodesic functions was first obtained by Nelson, Regge, and Zertuche~\cite{NelsonReggeZertuche1990}
 and is also known as the (semiclassical) reflection equation algebra, or twisted Yangian.

It is well-known that $\dim(\mathcal{T}_{g,s})=|E(\Gamma_n)|=3n-6$, (recall that $g=\lfloor \frac{n-1}{2}\rfloor$, $s=\operatorname{par}(n)$)
while $\dim(\mathcal{A}_n)=\frac{1}{2}n(n-1)$.
The number of independent Casimirs for $\mathcal{T}_{g,s}$ is~$s$ and the number of independent Casimirs for~$\mathcal{A}_n$ is~$\lfloor\frac{n}{2}\rfloor$.

For the case $n=3$, $\dim(\Tcc_{1,1})=\dim(\Acc_3)=3$. Both $\mathcal{T}_{1,1}$ and
$\mathcal{A}_3$ have one Casimir. Therefore, the dimension of a generic symplectic leaf in $\mathcal{T}_{1,1}$ and in
$\mathcal{A}_3$ is equal to $2$. Moreover, any loop in this case is a sequence of $\mathbb{Z}_2$-invariant arcs and is $\mathbb{Z}_2$-(anti-)symmetric. This observation shows that any geodesic function on $\Tcc_{1,1}$ is $\mathbb{Z}_2$-invariant. (Recall that any elliptic curve has $\mathbb{Z}_2$-symmetry.) For $n>3$ (equivalently, $s>1$ or $g>2$), there exist non-$\mathbb{Z}_2$ (anti-)invariant
loops and, hence, non-$\mathbb{Z}_2$ invariant geodesic functions.

\begin{Corollary} $\mathbf{R}_3\cong \mathcal{O}[\mathcal{T}_{1,1}]$. For $n>3$, $\mathbf{R}_n \subsetneq \mathcal{O}[\mathcal{T}_{g,s}]$.
\end{Corollary}

For the next step, let us recall the \emph{skein relation}: for any two ${\rm SL}_2$ matrices $A$ and $B$ one has%
\begin{equation*}
\operatorname{tr}(A)\operatorname{tr}(B)=\operatorname{tr}(AB)+\operatorname{tr}\bigl(AB^{-1}\bigr).
\end{equation*}
(see Figure~\ref{fig:skein}).

\begin{figure}[t]
\centering
\includegraphics[width=60mm]{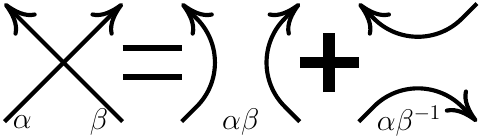}

\caption{Two paths $\alpha$ and $\beta$ intersect at one point as shown on the left. The product of the traces of corresponding monodromy operators satisfies the skein relation $\operatorname{tr}(M_\alpha)\operatorname{tr}(M_\beta)=\operatorname{tr}(M_\alpha M_\beta)+\operatorname{tr}\bigl(M_\alpha M_\beta^{-1}\bigr)=\operatorname{tr}(M_{\alpha\beta})+\operatorname{tr}(M_{\alpha\beta^{-1}})$.
Notice that path~$\alpha\beta$ is isotopic to the first path on the right-hand side and the path~$\alpha\beta^{-1}$ is isotopic to the second one.}\label{fig:skein}
\end{figure}

\begin{Theorem}\label{thm:quadraticextension} ${\bf R}_n\cong \mathcal{O}[\mathcal{T}_{g,s}]^{\mathbb{Z}_2}$.
\end{Theorem}

The proof of Theorem~\ref{thm:quadraticextension} occupies the rest of this section.

For convenience, we redraw the ribbon graph $\Gamma_n$ (Figure~\ref{fign5v1} for $n=5$) as a homeomorphic \emph{twisted ribbon graph} $\widetilde \Gamma_n$ (see Figure~\ref{fig:ribbontwist}).

\begin{Remark}\label{lem:twistedribbon}
The surface of the twisted ribbon graph $\widetilde\Gamma_5$ in Figure~\ref{fig:ribbontwist} is homeomorphic to the surface of the ribbon graph~$\Gamma_5$ in Figure~\ref{fign5v1}. Similarly, surfaces $\widetilde\Gamma_n$ are defined that are homeomorphic to~$\Gamma_n$.
\end{Remark}

\begin{figure}[t]
\centering
\includegraphics[width=60mm]{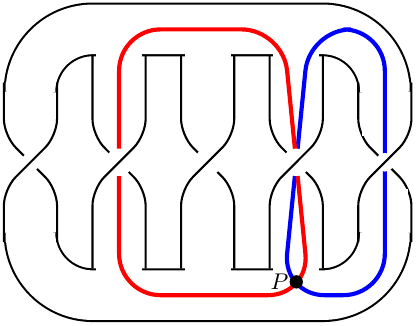}

\caption{The twisted ribbon graph $\widetilde\Gamma_n$ for $n=5$. The red curve is the loop~$\gamma_{24}$. The blue curve is the loop~$\gamma_{45}$. The two loops intersect at point~$P$.}\label{fig:ribbontwist}
\end{figure}

\begin{Remark}\label{rmrk:nonselfintersectinggenerators}
A collection of loops $\gamma_1,\dots,\gamma_m$ in $\widetilde\Gamma_n$ corresponds to the monomial $\prod_{\alpha=1}^m G_{\gamma_\alpha}$. The algebra of geodesic functions
has an additive spanning set consisting of monomials corresponding to collections of loops.
The skein relation expresses the geodesic function of a loop with self-intersections as a polynomial in geodesic functions of the loops with fewer self-intersections. By induction, we conclude that the algebra of geodesic functions has a spanning set consisting only of monomials corresponding only to collections of
{nonintersecting} loops
(laminations).
\end{Remark}

Note that the isotopy class of any oriented loop in $\widetilde\Gamma_n$ is entirely described by the cyclically ordered sequence of an even number of vertical intervals
traversed alternately upward and downward in the loop path.
We denote by $\widehat{i}$ (correspondingly, \smash{$\uwidehat{i}$}) the directed upwards (correspondingly, directed downwards) interval along the $i$th vertical edge of $\widetilde\Gamma_n$.
For instance, the~red loop in Figure~\ref{fig:ribbontwist} travelled clockwise corresponds to the sequence \smash{$\widehat{2}\uwidehat{4}$} (or, equivalently, \smash{$\uwidehat{4}\widehat{2}$}) while~the red loop travelled counterclockwise corresponds to the sequence \smash{$\uwidehat{2}\widehat{4}$}
(equivalently, \smash{$\widehat{4}\uwidehat{2}$}).
Note that any sequence describing a loop in~$\widetilde\Gamma_n$ has interlacing overhat and underhat indices. Therefore, such sequence is entirely determined by the direction of the first index in the sequence.
Reversing the orientation of the loop interchanges all overhats and underhats and reverses the cyclic order of the sequence.
We will denote the oriented closed loop by the sequence of indices, with the first index oriented up or down as $(\widehat{i_1} i_2\dots i_{2k-1} i_{2k})$. The corresponding geodesic function is denoted $G(\widehat{i_1} i_2\dots i_{2k-1} i_{2k})$.
Recall that the geodesic function does not depend on the orientation of the loop. If
reversing the orientation of the sequence produces
the sequence in the same cyclic equivalence
class as the original sequence, then the change from underhat to overhat and vice versa will not change the geodesic function.
In this situation, we will omit the overhat (underhat) if it does not lead to confusion.
 In particular, \smash{$G(i_1 i_2)=G(\widehat{i_1} i_2)=G(\uwidehat{i_1} i_2)$}.
Similarly, \smash{$G(i_1 i_2 i_3 i_2)=G(\widehat{i_1}i_2 i_3 i_2)=G(\uwidehat{i_1} i_2 i_3 i_2)$}.

To avoid ambiguity, we note that in the class of sequences up to cyclic permutations, there is a unique lexicographically minimal sequence, which we will use to label the corresponding loop.

\begin{Lemma}\label{lem:GijINR} For all $1\le i<j\le n$, $G_{ij}=G(ij)\in \mathbf{R}_n$ and
 for all $1\le i<j<k\le n$ elements $G(ijkj),G(ijik),G(ikjk)\in\mathbf{R}_n$.
\end{Lemma}
\begin{proof}
 For any $1\le i_1<i_2<i_3\le n$, the skein relation implies $G(i_1 i_2)\cdot G(i_2 i_3)=G(i_1 i_3)+G(i_1 i_2 i_3 i_2)$. Note that the loops $(i_1 i_2)$ and $(i_2 i_3)$ intersect exactly in one point implying that
$\{G(i_1 i_2),G(i_2 i_3)\}=\frac{1}{2} G(i_1 i_3)-\frac{1}{2}G(i_1 i_2 i_3 i_2)$ by~(\ref{eq:Goldman}).
Introduce $\langle\langle f,g\rangle\rangle:=\frac{1}{2} f\cdot g + \{f,g\}$. Observe that $\langle\langle G(i_1 i_2),G(i_2 i_3)\rangle\rangle = G(i_1 i_3)$, hence, $G(i_1 i_3)\in \mathbf{R}_n$.
That immediately implies that $G(ijkj)\in\mathbf{R}_n$. Similarly, ${G(ijik),G(ikjk)\in\mathbf{R}_n}$.
\end{proof}

\begin{Remark} For an algebra $\mathbb{A}$, we consider its quadratic extension $\mathbb{A}[\xi]$, where $\xi$ is a~root of a~quadratic polynomial that is irreducible over $\mathbb{A}$,
$p(\xi)=a_0+a_1\xi+a_2\xi^2$, $a_i\in\mathbb{A}$, $p(\xi)=0$.
\end{Remark}

\begin{Lemma}\label{lem:GijklINSqrtR} For $1\le i_1<i_2<i_3<i_4\le n$, the algebra $\mathbf{R}_n[G(\widehat{i_1} i_2 i_3 i_4)]$ is a quadratic extension of $\mathbf{R}_n$.
\end{Lemma}
\begin{proof} Let $f=G(\widehat{i_1}i_2 i_3 i_4)$, \smash{$g=G(\uwidehat{i_1}i_2 i_3 i_4)$} and $\sigma_1=f+g$ and $\sigma_2=fg$ be the two elementary symmetric functions of $f$ and $g$. Construct an unipotent upper-triangular $4\times 4$ matrix $U$ with $U_{st}=G(i_s i_t)\in\mathbf{R}_n$ for all $1\le s<t\le 4$. Coefficient of the first power of $\lambda$ of $\det\bigl(U+\lambda U^{\mathsf{T}}\bigr)$ coincides with $\sigma_1$. The Pfaffian of the skew-symmetric matrix $U- U^{\mathsf{T}}$ is
\[
{\rm Pf}\bigl(U-U^{\mathsf{T}}\bigr)=\sqrt{\det\bigl(U-U^{\mathsf{T}}\bigr)}
\]
and $\sigma_1={\rm Pf}\bigl(U-U^{\mathsf{T}}\bigr)$. Then, $\sigma_1,\sigma_2\in\mathbf{R}_n$. It remains to notice that $f,g\not\in\mathbf{R}_n$.
\end{proof}

\begin{Corollary}\label{cor:1234INR} $G(\uwidehat{i_1}i_2 i_3 i_4)\in \mathbf{R}_n[G(\widehat{i_1}i_2 i_3 i_4)]$ and $G(\widehat{i_1}i_2 i_3 i_4)\in \smash{\mathbf{R}_n[G(\uwidehat{i_1}i_2 i_3 i_4)]}$.
\end{Corollary}
\begin{proof} As in the proof of Lemma~\ref{lem:GijklINSqrtR}, $\sigma_1\in\mathbf{R}_n$, which implies the statement.
\end{proof}

Let $\Rcc_n=\mathbf{R}_n[G(\widehat{1}234)]$.

\begin{Lemma}\label{INR}
For all $1\le i_1<i_2<i_3<i_4\le n$ both $G(\widehat{i_1}i_2 i_3 i_4)$ and $G(\uwidehat{i_1}i_2 i_3 i_4)$ belong to $\Rcc_n$.
\end{Lemma}
\begin{proof} For $k_4>i_4$, $\langle\langle G(\widehat{i_1}i_2 i_3 i_4),G(i_4 k_4)\rangle\rangle=G(\widehat{i_1}i_2 i_3 k_4)$. Since $G(i_4 k_4)\in\mathbf{R}_n$, we obtain that~$G(\widehat{i_1}i_2 i_3 i_4)$ and $G(\widehat{i_1}i_2 i_3 k_4)$ either both belong or not belong to $\Rcc_n$ simultaneously. We claim a~similar statement for $G(\widehat{i_1}i_2 i_3 i_4)$ and $G(\widehat{i_1}i_2 k_3 i_4)$ if $i_3<k_3<i_4$;
 $G(\widehat{i_1}i_2 i_3 i_4)$ and $G(\widehat{i_1}k_2 i_3 i_4)$ if $i_2<k_2<i_3$;
and
 $G(\widehat{i_1}i_2 i_3 i_4)$ and $G(\widehat{k_1}i_2 i_3 i_4)$ if $k_1<i_1$. Combining all these statements above with the fact that $G(\widehat{1}234)\in\Rcc_n$, we conclude that
 $G(\widehat{i_1}i_2 i_3 i_4)\in\Rcc_n$. By Corollary~\ref{cor:1234INR}, we see also that \smash{$G(\uwidehat{i_1}i_2 i_3 i_4)\in\Rcc_n$}.
\end{proof}

\begin{Lemma}\label{INR2} The following statements hold:
\begin{enumerate}\itemsep=0pt
\item[$(1)$] For all $1\le i_1<i_2<i_3<i_4<i_5\le n$, $G(\widehat{i_1}i_3 i_5 i_4 i_3 i_2)\in \Rcc_n$.
\item[$(2)$] For all $1\le i_1<i_2<i_3<i_4<i_5<i_6\le n$, $\smash{G(\uwidehat{i_1}i_2 i_3 i_4 i_5 i_6)}\in \Rcc_n$. Similarly, $G(\widehat{i_1} i_2 i_3 i_4 i_5 i_6)\in\Rcc_n$.
\item[$(3)$] For all $1\le i_1<i_2<i_3<i_4<i_5<i_6\le n$, \smash{$G(\uwidehat{i_1}i_2 i_3 i_6 i_5 i_4)\in \Rcc_n$}. Equivalently, $G(\widehat{i_1} i_6 i_5 i_4 i_3 i_2)\in\Rcc_n$.
\end{enumerate}
\end{Lemma}

\begin{proof}
$(1)$ Using the skein relations, we see that $2G(\widehat{i_1} i_4 i_3 i_2)\cdot G(i_3 i_5)-2G(\widehat{i_1} i_5 i_3 i_2)\cdot G(i_3 i_4)
=G(\widehat{i_1} i_3 i_5 i_4 i_3 i_2)+G(\widehat{i_1} i_5 i_3 i_4 i_3 i_2)-G(\widehat{i_1} i_5 i_4 i_2)-G(\widehat{i_1} i_5 i_3 i_4 i_3 i_2)
=G(\widehat{i_1} i_3 i_5 i_4 i_3 i_2)-G(\widehat{i_1} i_5 i_4 i_2)$.

Therefore, $G(\widehat{i_1} i_3 i_5 i_4 i_3 i_2)=2G(\widehat{i_1} i_4 i_3 i_2)\cdot G(i_3 i_5)-2G(\widehat{i_1} i_5 i_3 i_2)\cdot G(i_3 i_4)+G(\widehat{i_1} i_5 i_4 i_2)$.
It remains to notice that the right-hand side of the last equality belongs to $\Rcc_n$ by Lemmas~\ref{lem:GijINR} and~\ref{INR}.

$(2)$ $G(\uwidehat{i_1} i_2 i_3 i_4 i_5 i_6)=-2 G(\uwidehat{i_1} i_2 i_4 i_6)\cdot G(i_3 i_5)+\frac{1}{2}G(\uwidehat{i_1} i_2 i_3 i_6)\cdot G(i_4 i_5)+\frac{1}{2} G(\uwidehat{i_1}i_2 i_5 i_6)\cdot G(i_3 i_4)+\{G(i_3 i_5), G(\uwidehat{i_1} i_2 i_4 i_6)\}$. By Lemmas~\ref{lem:GijINR}--\ref{INR}, the right-hand side belongs to $\Rcc_n$. The second statement is proved similarly.

$(3)$ $G(\uwidehat{i_1}i_2 i_3 i_6 i_5 i_4)=G(\uwidehat{i_1} i_2 i_3 i_4)\cdot G(i_5 i_6)+G(\uwidehat{i_1} i_2 i_3 i_6)\cdot G(i_4 i_5)-G(\uwidehat{i_1}i_2 i_3 i_4 i_5 i_6)-
G(\uwidehat{i_1} i_2 i_3 i_5)\cdot G(i_4 i_6)$. The right-hand side belongs to $\Rcc_n$ by Lemmas~\ref{lem:GijINR}--\ref{INR} above and part~(2).
\end{proof}

\begin{Definition} We call the union of two disconnected horizontal intervals of $\widetilde\Gamma_n$ between vertical edge $i$ and $i+1$ {\it gate} $i$.
We say that a simple loop in $\widetilde\Gamma_n$ is {\it horizontally $2$-bounded} if it travels every gate and every vertical edge at most twice: at most once in one direction and at most once in the opposite direction.

We say that a simple loop in $\widetilde\Gamma_n$ is {\it $2$-bounded} if it travels every edge (vertical and horizontal) at most twice: at most once in one direction and at most once in the opposite direction.
\end{Definition}

\begin{Lemma} \label{lem:horizontally2bounded}
If $\gamma$ is a simple horizontally $2$-bounded loop in $\widetilde\Gamma_n$ then $G_\gamma\in \Rcc_n$.
\end{Lemma}
\begin{proof} The proof follows the arguments of Lemmas~\ref{INR} and~\ref{INR2} for loops with a longer tail.
\end{proof}

\begin{Lemma} \label{lem:skippedleg} Let $\gamma$ be a simple $2$-bounded loop that travels once left-to-right along the bottom horizontal legs in gates $i-1$ and $i$, once right-to-left along the top horizontal legs
in the same gates and does not travel along the vertical edge $i$ $($as shown in Figure~{\rm \ref{fig:skippedleg})}.
Denote by $\gamma_L$ $($or,~$\gamma_R)$ the shortened left $($or, right$)$ part of $\gamma$ $($as shown in Figure~{\rm \ref{fig:leftrightparts})}. If both~$G_{\gamma_L}$ and~$G_{\gamma_R}$ belong to~$\Rcc_n$, then $G_\gamma\in\Rcc_n$.
\end{Lemma}

\begin{figure}[t]
\centering
\includegraphics[width=60mm]{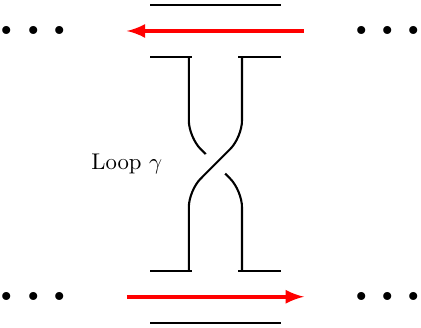}

\caption{The simple closed loop $\gamma$ (in red) skips the vertical leg $i$.}\label{fig:skippedleg}
\end{figure}

\begin{figure}[t]
\centering
\includegraphics[width=72mm]{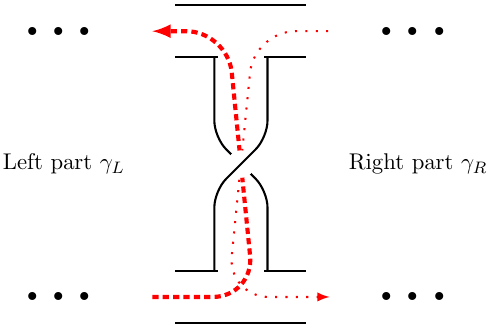}

\caption{The left part $\gamma_L$ is thick dashed, the right part $\gamma_R$ is thin dotted.}\label{fig:leftrightparts}
\end{figure}

\begin{proof} As in the proof of Lemma~\ref{lem:GijINR}, it is enough to notice that $\langle\langle G_{\gamma_L},G_{\gamma_R}\rangle\rangle=G_\gamma$.
\end{proof}

\begin{Remark} Note a simple $2$-bounded loop that skips vertical leg $i$ and travels once along the top and the bottom horizontal legs of gates~$i$ and~$i+1$ must travel along the top and bottom horizontal legs in the opposite directions, otherwise it cannot be a closed loop.
\end{Remark}

\begin{Definition} We call a simple $2$-bounded loop {\it maximal} if it travels every edge exactly twice: once in each direction.
\end{Definition}

\begin{Lemma}\label{lem:maxloop} For all odd $n\ge 3$, a maximal simple $2$-bounded loop in $\widetilde\Gamma_n$ is either $(\widehat{1}\ 2\ 3\ \dots n-1\ n \ 1\ 2\dots n-1\ n )$ or $(\uwidehat{1}\ 2\ 3\ \dots n-1\ n\ 1\ 2\dots n-1\ n )$
\end{Lemma}
\begin{proof} By inspection.
\end{proof}

\begin{Corollary} A maximal simple $2$-bounded loop exists only for odd $n$. For even $n$, both sequences of indices in Lemma~{\rm \ref{lem:maxloop}} lead to two connected components.
\end{Corollary}

\begin{Lemma}\label{lem:maxloopINR} If $n$ is odd and $\gamma$ is the maximal loop then $G_\gamma\in\mathbf{R}_n$.
\end{Lemma}
\begin{proof}
Let
\begin{align*}
& G_{123\dots n}=G_{12} G_{23} G_{34}\dots G_{n-1,n} G_{1n},\qquad
G^{(2)}=\sum_{1\le i<j\le n} G_{ij}^2, \\
& G^{(3)}=\sum_{1\le i<j<k\le n} G_{ij}G_{jk}G_{ik}, \qquad
G^{(4)}=\sum_{1\le i<j<k<l\le n} G_{ij}G_{jk}G_{kl}G_{il},
\end{align*}
 etc.
Define the $n$th Markov function $\mathcal{M}^{(n)}$ on $\mathcal{T}_{g,s}$ as follows:
\[
\mathcal{M}^{(n)}=G_{123\dots n}-G^{(n-1)}+G^{(n-2)}-\dots+(-1)^n G^{(2)} +(-1)^{n+1} 2(2-n)\in\mathbf{R}_n.
\]
For odd $n$, the $\mathbb{Z}_2$-symmetry reverses the orientation of the loop $(\widehat{1}\ 2\ 3\ \dots n-1\ n \ 1\ 2\dots n-1\ n )$, implying equality $G(\widehat{1}\ 2\ 3\ \dots n-1\ n \ 1\ 2\dots n-1\ n )=G(\uwidehat{1}\ 2\ 3\ \dots n-1\ n\ 1\ 2\dots n-1\ n )$. Repeatedly applying the skein relation to the product
$G_{12} G_{23} G_{34}\dots G_{n-1,n} G_{1n}$ and all products in
$G^{(k)}$, $k=2,\dots,n-1$, we observe that $G(\widehat{1}\ 2\ 3\ \dots n-1\ n \ 1\ 2\dots n-1\ n )=\mathcal{M}^{(n)}$.
\end{proof}

\begin{Lemma}\label{lem:2bnddINR}
For any simple $2$-bounded loop $\gamma$, its geodesic function $G_\gamma$ belongs to $\Rcc_n$.
\end{Lemma}
\begin{proof}
The claim follows from Lemmas~\ref{lem:horizontally2bounded},~\ref{lem:skippedleg}, and~\ref{lem:maxloopINR}.
\end{proof}

\begin{Remark}\label{rmrk:twice} Any simple non-$2$-bounded loop in $\widetilde\Gamma_n$ has at least one edge that it travels twice in the same direction (segments $s_1\rightarrow t_1$ and $s_2\rightarrow t_2$ in Figure~\ref{fig:twice} both run in the same direction along the ribbon edge $e$). Other segments of the same loop may run either between segments $s_1\to t_1$ and $s_2\to t_2$ or above or below in either direction, as the dashed one in Figure~\ref{fig:twice}. Since both $s_1\to t_1$ and $s_2\to t_2$ are nonintersecting segments of the same simple loop $\gamma$, this loop follows from $s_1$ to $t_1$, then to $s_2$, then to~$t_2$, and then returns to~$s_1$.
\end{Remark}

\begin{figure}[t]
\centering
\includegraphics[width=65mm]{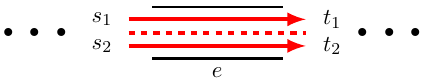}

\caption{An edge traveled twice in the same direction.}\label{fig:twice}
\end{figure}

\begin{Theorem}\label{thm:squareroot} For any simple loop $\gamma$, its geodesic function~$G_\gamma$ belongs to~$\Rcc_n$.
\end{Theorem}

\begin{proof} For each edge $e$ of the ribbon graph we define the \emph{lane number} of $e$ in $\gamma$ as
$\ell_e(\gamma)=\max(0,r_{\lvec{e}}(\gamma)-1)+ \max(0,r_{\rvec{e}}(\gamma)-1)$ where
$r_{\lvec{e}}(\gamma)$ is the number of times $\gamma$ travels $e$ in one direction while
$r_{\rvec{e}}(\gamma)$ is the number of times $\gamma$ travels $e$ in the opposite direction.
Define \emph{the total lane number} of $\gamma$ as $\ell(\gamma)=\sum_e \ell_e(\gamma)$ where the sum is over all edges of $\widetilde\Gamma_n$.

We prove the statement by induction on $\ell(\gamma)$.

Lemma~\ref{lem:2bnddINR} is the base case of the induction because $\ell(\gamma)=0$ if and only if $\gamma$ is a $2$-bounded simple loop.

For the induction step, we consider $e$ as in Figure~\ref{fig:twice}. Note that if we replace segments $s_1\to t_1$ and $s_2\to t_2$ by the crossed segments $s_1\to t_2$ and $s_2\to t_1$ as in Figure~\ref{fig:crossing}, we obtain two intersecting loops $\gamma_1$ and $\gamma_2$ (not necessarily simple). Then we can use the skein relation for the product $G_{\gamma_1}\cdot G_{\gamma_2}$ to obtain the sum of two geodesic functions of (possibly, nonsimple) loops: one is $\gamma$ and the second $\tilde\gamma$ whose total lane number is strictly smaller $\ell(\tilde\gamma)<\ell(\gamma)$. It remains to notice that $\ell(\gamma_1)<\ell(\gamma)$ and $\ell(\gamma_2)<\ell(\gamma)$ and also that resolving a crossing using the skein relation does not increase the total lane number. Therefore, we express $G_\gamma$ through $G_{\gamma_1}, G_{\gamma_2}$ and $G_{\tilde\gamma}$ with these geodesic functions belonging to $\Rcc_n$ by the induction assumptions.

\begin{figure}[t]
\centering
\includegraphics[width=65mm]{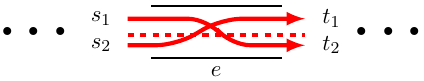}

\caption{Crossing the edges travelled in the same direction.}\label{fig:crossing}
\end{figure}

That completes the proof of Theorem~\ref{thm:squareroot}.
\end{proof}

Now we are ready to prove Theorem~\ref{thm:quadraticextension}.

\begin{proof}[Proof of Theorem~\ref{thm:quadraticextension}] Theorem~\ref{thm:squareroot} and Remark~\ref{rmrk:nonselfintersectinggenerators} imply that the geodesic function of any loop is in the quadratic extension $\Rcc_n$ of $\mathbf{R}_n$. Let $P$ be a quadratic polynomial in $\mathbf{R}_n[x]$ with roots $\rho_1$ and $\rho_2$ such that $\Rcc_n=\mathbf{R}_n[\rho_1]=\mathbf{R}_n[\rho_2]$.
$\mathbb{Z}_2$-action permutes the roots $\rho_1$ and $\rho_2$. In view of Lemma~\ref{lem:GijklINSqrtR}, elementary symmetric functions of $\rho_1$ and $\rho_2$ belong to $\mathbf{R}_n$. It implies that $\mathbf{R}_n$ coincides with the $\mathbb{Z}_2$-invariant subalgebra $\mathcal{O}[\mathcal{T}_{g,s}]^{\mathbb{Z}_2}$
\end{proof}

Since $\mathcal{T}_{g,s}$ is a union of isomorphic symplectic leaves and, by Theorem~\ref{thm:quadraticextension}, $\varphi$ is a $2$-to-$1$ Poisson map, ${\rm Im}(\varphi)$ is a union of
symplectic leaves of the same dimension in~$\mathcal{A}_n$. Moreover, if for $s=2$ we consider a generic (not containing fixed points of ${\mathbb Z}_2$-action) geometric symplectic leaf of the Bondal Poisson structure in ${\mathcal A}_n$ it is locally isomorphic to the generic symplectic leaf of $\mathcal{T}_{g,s}$.

\begin{Definition} We call ${\rm Im}(\varphi)$ a \emph{geometric locus}.
 A \emph{geometric symplectic leaf} is a generic symplectic leaf in the geometric locus.
\end{Definition}

We now consider low dimensional cases. For $n=3$, $\dim(\mathcal T_{1,1})=\dim({\mathcal A}_3)=3$.
For $n=4$, $\dim(\Tcc_{1,2})=\dim(\Acc_4)=6$.

If $n=3$ or $n=4$, since the dimensions coincide, ${\rm Im}(\varphi)$ is a subset of $\mathcal{A}_n$ of maximal dimension.
The dimension of a generic symplectic leaf in $\Tcc_{1,2}$ and in $\Acc_4$ is~$4$.
The maps from the clusters of $\mathcal A_3$ and $\mathcal A_4$ quivers to the clusters in the corresponding Teichm\"uller spaces are identities since both $\mathcal A_3$ and $\mathcal A_4$ quivers coincide with the quivers in Teichm\"uller spaces. Therefore, the positivity condition merely requires all $\mathcal A_3$ and $\mathcal A_4$ cluster variables to be positive.

\begin{figure}[t]
\centering
\includegraphics[width=25mm]{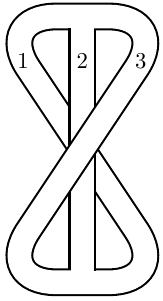} \qquad
\raisebox{22mm}{$
 \begin{bmatrix}
 1 & G_{12} & G_{13}\\
 0 & 1 & G_{23}\\
 0 & 0 & 1
 \end{bmatrix}
 $}

\caption{$n=3$. Edge $k$ is assigned an exponential shear parameter $x_k$. See Figure~\ref{fig:hol}.}\label{fign3}

\end{figure}

\begin{figure}[t]
\centering
\includegraphics[width=37mm]{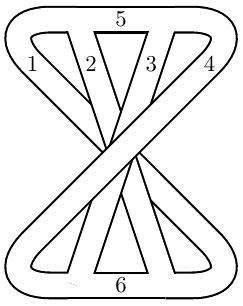} \qquad
\raisebox{21mm}{$
\begin{bmatrix}
 1 & G_{12} & G_{13} & G_{14}\\
 0 & 1 & G_{23} & G_{24}\\
 0 & 0 & 1 & G_{34}\\
 0 & 0 & 0 & 1
\end{bmatrix} $}

\caption{$n=4$. Edge $k$ is assigned an exponential shear parameter~$x_k$.}\label{fign4}

\end{figure}
\begin{figure}[t]
\centering
\includegraphics[width=45mm]{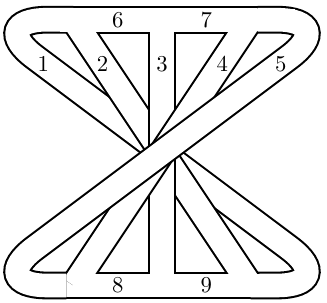} \qquad
\raisebox{19mm}{$
\begin{bmatrix}
 1 & G_{12} & G_{13} & G_{14} & G_{15}\\
 0 & 1 & G_{23} & G_{24} & G_{25}\\
 0 & 0 & 1 & G_{34} & G_{35}\\
 0 & 0 & 0 & 1 & G_{45}\\
 0 & 0 & 0 & 0 & 1
\end{bmatrix} $}

\caption{$n=5$. Edge $k$ is assigned an exponential shear parameter $x_k$.}
\label{fign5}

\end{figure}

For $n\geq 5$, $\dim(\Tcc_{g,s})<\dim(\Acc_n)$. For example, for $n=5$, $\dim(\Tcc_{2,1})=9$ and $\dim(\Acc_5)=10$.
${\rm Im}(\varphi)$ is a subvariety of~$\Acc_5$ of codimension~1. The dimensions of generic symplectic leaves in~$\Tcc_{2,1}$ and in~$\Acc_5$ still coincide, and are equal to~$8$.

For $n\ge 6$, the dimension of a generic symplectic leaf in $\Tcc_{g,s}$ is strictly less than that of a~generic symplectic leaf in~$\Acc_n$.

\begin{Remark} In the cases above, the functions~$G_{ij}$ are functions of the exponential shear parameters~$x_k$ of the edges~$k$ of the ribbon graphs in Figures~\ref{fign3}--\ref{fign5}.
\end{Remark}

\section{Description of pre-geometric leaf: Rank condition}\label{sec:rankcondition}

In this and the next sections, we set aside the positivity condition, concentrating on the necessary rank condition.

For
\[
A=\begin{pmatrix}1&&&G_{ij}\\&1&&\\&&\ddots&\\0&&&1\end{pmatrix}\in \operatorname{Im}(\varphi),
\]
 we have
\[
A+A^{\mathsf{T}}=\begin{pmatrix}2&&&G_{ij}\\&2&&\\&&\ddots&\\G_{ij}&&&2\end{pmatrix}.
\]
Note that $\gamma_{ii}$ is a contractible loop. Hence $G_{ii}=2$ and we can replace the $i$th diagonal element of $A+A^{\mathsf{T}}$ by $G_{ii}$ for $1\leq i\leq n$.

We write
\begin{align*}
A+A^{\mathsf{T}}=\begin{pmatrix}G_{11}&&&G_{ij}\\&G_{22}&&\\&&\ddots&\\G_{ji}&&&G_{nn}\end{pmatrix},
\end{align*}
where $G_{ji}=G_{ij}$.

Choose points $S_1$ and $S_2$ in a fat graph as shown in Figure~\ref{fig5}.
For any $i$, we introduce the holonomy operator $F_i\in {\rm SL}_2(\mathbb{R})$ defined as a product of elementary matrices along the path connecting $S_1$ to $S_2$ and containing edge $i$, as drawn in Figure~\ref{fig5}.
The monodromy is $M_{\gamma_{ij}}=F_j^{-1}F_i$.

\begin{figure}[t]
\centering
\includegraphics[width=35mm]{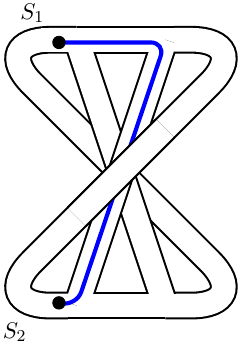}

\caption{Path corresponding to $F_3=L X(x_5) R X(x_3) R$.}\label{fig5}

\end{figure}

\begin{Definition}
Define a map from ${\rm SL}_2\times {\rm SL}_2$ to $\R$ by
 \begin{align*}
( A,B) := {\rm tr}\bigl(B^{-1}A\bigr).
 \end{align*}
\end{Definition}
Note that ${\rm SL}_2$ is a subset of the linear space of $2\times 2$ real matrices ${\rm Mat}_{2\times 2}\cong \mathbb{R}^4$ whose linear span is ${\rm Mat}_{2\times 2}$. Since $\det(B)=1$, the operator $B\mapsto B^{-1}$ is a restriction of a linear operator
${\rm Mat}_{2\times 2}\to {\rm Mat}_{2\times 2}$ implying that
 $ (\bullet,\bullet )$ allows an extension to a bilinear $\mathbb{R}$-valued form on ${\rm Mat}_{2\times 2}$. By an abuse of notation, we also denote the bilinear form by $(\bullet,\bullet)$.

{\samepage For $A, B, C \in {\rm SL}_2$, $\alpha, \beta \in \R$, we have
\begin{itemize}\itemsep=0pt
 \item $( A,B)={\rm tr}\bigl(B^{-1}A\bigr)={\rm tr}\bigl(A^{-1}B\bigr)= ( B, A )$.
 \item $ ( \alpha A+\beta B, C )= {\rm tr}\bigl(C^{-1}(\alpha A+\beta B)\bigr)= \alpha {\rm tr}\bigl(C^{-1}A\bigr)+\beta {\rm tr}\bigl(C^{-1}B\bigr)=\alpha ( A, C ) + \beta ( B, C )$.
 \item $( A, A) = {\rm tr}\bigl(A^{-1}A\bigr) = {\rm tr}({\rm Id})=2$.
\end{itemize}}%
Hence the bilinear form $\langle \cdot, \cdot \rangle$ restricted to ${\rm SL}_2\times {\rm SL}_2$ is a restriction of the inner product on ${\rm Mat}_{2\times 2}$.
The $(i,j)$-th entry of the matrix $A+A^{\mathsf{T}}$ is ${\rm tr}\bigl(F_j^{-1}F_i\bigr)= ( F_i, F_j )$, i.e., $A+A^{\mathsf{T}}$ is the Gram matrix of inner products $( F_i, F_j)$.

Recall that the Gram matrix of $n$ vectors ${\bf v}_1, \dots, {\bf v}_n$ in $m$-dimensional space $\R^m$ equipped with a bilinear form $\big(\bullet , \bullet\big)$ is $H=(h_{ij})_{i,j=1}^n$ with $h_{ij} = \big( {\bf v}_i, {\bf v}_j \big)$.

The following classical statement is well known.

\begin{Lemma}
${\rm rank}(H)\leq m$.
\end{Lemma}

Since all $F_i$, $i=1,\dots,n$ are vectors in $M_{2\times 2}\cong \R^4$, by the previous lemma, ${\rm rank}\bigl(A+A^{\mathsf{T}}\bigr)\leq 4$.

This proves the following condition on the matrix $A$ (see \cite{NelsonRegge1989,NelsonReggeZertuche1990}).
\begin{Corollary}[rank condition]\label{cor:rank}
\label{lem:rankcondition}
 If $A\in{\rm Im}(\varphi)$, then ${\rm rank}\bigl(A+A^{\mathsf{T}}\bigr)\leq 4$.
\end{Corollary}

 \begin{Conjecture} \label{con:rank} We conjecture that a positivity condition (described explicitly for $n=5,6$ in the next sections) and the rank condition are necessary and sufficient for $A$ to belong to the geometric locus.
 \end{Conjecture}

 The results of this paper imply that the conjecture holds for $n\le 6$.

\section{Solving the rank condition}\label{sec:solutionRank}

The strategy of verifying Conjecture~\ref{con:rank} is as follows: we consider all possible solutions of the rank condition in the $\mathcal{A}_n$ quiver with cluster variables $(x_\alpha)$, transforming them, if necessary, using standard mutations to a new quiver with cluster variables $(\widetilde x_\beta)$. The rank condition becomes a~system of equations $f_m(\widetilde x_\beta)=0$. Constraints $f_m$ are Lagrangian (of the first kind) because the rank condition is
coisotropic. We then resolve the rank condition on the set of geodesic functions~$G_{i,j}(\widetilde x_\beta)$ and show that these functions can be written in terms of new coordinates~$(y_\gamma)$ (in all cases of nontrivial reduction,
the number of $y$'s is strictly smaller than the number of $\widetilde x_\beta$), that is, $G_{i,j}(\widetilde x_\beta)=G_{i,j} (y_\gamma(\widetilde x_\beta) )$. The new variables $y_\gamma(\widetilde x_\beta)$ possess a Poisson structure induced by that of $\widetilde x_\beta$.
If Poisson relations for $y_\gamma$'s are given by equation~\eqref{eq:Gldmn}, for some ribbon graph of the genus $\lfloor \frac{n-1}{2}\rfloor$ with $\operatorname{par}(n)$ holes then $G_{i,j}(y_\gamma)$ become geodesic functions from Section~\ref{sec:shear}.

\subsection{Expressing the rank condition in terms of transport matrices}

For a generic unipotent upper-triangular matrix $A_n$, ${\rm rank}\bigl(A_n+A^{\mathsf{T}}_n\bigr)=n$. If ${\rm corank}\bigl(A_n+A^{\mathsf{T}}_n\bigr)\ge 1$, then $\det \bigl(A_n+A^{\mathsf{T}}_n\bigr)= 0$.
The following \emph{groupoid condition} was shown in~\cite{10.1093/imrn/rnac101}:
\begin{equation*}
M_2=M_3 M_1,
\end{equation*}
where $M_i$, $D_i$ and $S$ are defined by equation~\eqref{eq:DandS}.

Using that $\hat A_n=M_1^{\mathsf{T}} M_2$, it was observed in \cite{chekhov2022roots} that
\begin{align*}
& \hat A_n= M_1^{\mathsf{T}}M_2= M_1^{\mathsf{T}} M_3M_1 = (S\mathcal M_1)^{\mathsf{T}}(S\mathcal M_3)(S \mathcal M_1)\bigl(D_1^2 D_3\bigr)^{-1}\\
& \hphantom{\hat A_n}{}
= \mathcal M_1^{\mathsf{T}} \bigl(S^{\mathsf{T}} S\bigr) \mathcal M_3 S \mathcal M_1\bigl(D_1^2 D_3\bigr)^{-1}
 = \mathcal M_1^{\mathsf{T}} \mathcal M_3 S \mathcal M_1  \bigl(D_1^2 D_3\bigr)^{-1}, \\
&\hat A_n + \hat A_n^{\mathsf{T}}=  \bigl(D_1^2 D_3\bigr)^{-1}\bigl( \mathcal M_1^{\mathsf{T}} \mathcal M_3 S \mathcal M_1 + \mathcal M_1^{\mathsf{T}} S^{\mathsf{T}} \mathcal M_3 \mathcal M_1\bigr)\\
 &
\hphantom{\hat A_n + \hat A_n^{\mathsf{T}}}{}
 =  \bigl(D_1^2 D_3\bigr)^{-1}\bigl(\mathcal M_1^{\mathsf{T}} \bigl(\mathcal M_3 S + S^{\mathsf{T}} \mathcal M_3^{\mathsf{T}}\bigr) \mathcal M_1\bigr).
\end{align*}

Since $\mathcal{M}_1$ is nondegenerate, and $D_1$ and $D_3$ are nonzero,
\begin{equation}\label{eq:rank-AM}
{\rm rank} \bigl(\hat A_n+\hat A_n^{\mathsf{T}}\bigr)={\rm rank}(\Psi),
\end{equation}
where
\begin{equation*}
 \Psi:=\mathcal M_3 S + S^{\mathsf{T}} \mathcal M_3^{\mathsf{T}}
\end{equation*}
is a lower anti-triangular matrix

\subsection[n = 3 or 4]{$\boldsymbol{n = 3}$ or $\boldsymbol{4}$}
In these two cases, the rank condition (Corollary~\ref{cor:rank}) holds automatically because $n\le 4$.
Moreover, the quivers $\mathcal A_3$ and $\mathcal A_4$ coincide with the quivers in Teichm\"uller spaces.
They are the quivers associated with ribbon graphs for genus-one surfaces with one and two holes, respectively.
The geometric locus is also characterized by the positivity condition. The positivity condition in both cases requires that all cluster variables in ${\mathcal A}_3$ (${\mathcal A}_4$, correspondingly) are positive.

\subsection[n = 5]{$\boldsymbol{n = 5}$}

When $n=5$, the rank condition holds if and only if $\det\bigl(A_5+A_5^{\mathsf{T}}\bigr)=0$.
By~(\ref{eq:rank-AM}), this means that at least one of the anti-diagonal terms of the matrix $\Psi$ must be zero. Let $(-1)^k m_{k,6-k}$, ${k=1,\dots,5}$ be anti-diagonal terms of $\mathcal M_3S$.
Recall that in Section~\ref{sec:logcanonicalgroupoid} the entries of $\mathcal M_3$ are expressed in terms of $Z_{ijk}$ or, equivalently, in terms of cluster parameters $b_1=Z_{401}Z_{140}$, ${b_2=Z_{131}}$, $b_3=Z_{122}$, $b_4=Z_{113}$, $b_5=Z_{104}Z_{410}$, $a_1=Z_{311}$, $a_2=Z_{302}Z_{230}$, $a_3=Z_{221}$, $a_4=Z_{212}$, $a_5=Z_{203}Z_{320}$ assigned to the vertices of the amalgamated quiver in Figure~\ref{quiverA5}.
Let us recall now that all antidiagonal terms of $\mathcal M_3S$ are monomials in $a_1,\dots,a_5, b_1, \dots, b_5$ and $\Psi^{\mathsf{T}}=\Psi$.

\begin{figure}[t]
\centering
\includegraphics[width=40mm]{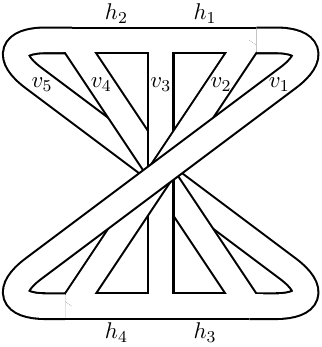}

 \caption{The labeled ribbon graph for $n=5$.}\label{fatgraph}

\end{figure}

In particular, we distinguish the following three cases depending on the position of the vanishing anti-diagonal element of $\Psi$:
\begin{enumerate}\itemsep=0pt
\item[(0)] $2m_{33}=0$. This case is empty because $m_{33}$ is a monomial in cluster variables $a_i$, $b_j$ and we assume that cluster variables never vanish;
\item[(1)] $m_{2,4}+m_{4,2}=m_{2,4}(1+C_1)=0$, where $C_1=\prod_{i=1}^5 a_i$. In this case, the rank condition is equivalent to the equation $C_1=-1$.
\item[(2)] $m_{1,5}+m_{5,1}=m_{1,5}(1+C_1C_2)=0$, where $C_1$ is as above and $C_2=\prod_{i=1}^5 b_i$. In this case, the rank condition becomes $C_1C_2=-1$.
\end{enumerate}

A geometric cluster quiver for $\mathcal{T}_{2,1}$ can be obtained from the ribbon graph in Figure~\ref{fatgraph}.
To simplify further formulas, we denote the exponential shear parameters of horizontal and slanted edges in Figure~\ref{fatgraph} by different letters.

The Goldman Poisson structure is described by a geometric quiver as explained in Section~\ref{sec:geometricleafdimension}. Each edge of the ribbon graph in Figure~\ref{fatgraph} corresponds to a vertex of the quiver (Figure~\ref{quiverTgs}) via equation~\eqref{eq:Gldmn}.
 An arrow of the quiver connects two edges~$a$ to~$b$ that share a common endpoint if $b$ immediately follows $a$ under rotation in the anticlockwise direction about the common endpoint.

\begin{figure}[t]
\centering
\includegraphics[width=40mm]{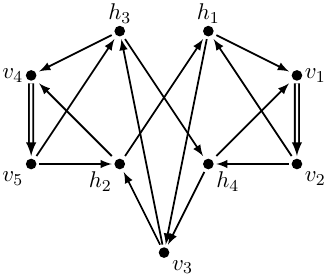}

\caption{Quiver for $\mathcal{T}_{2,1}$.}\label{quiverTgs}

\end{figure}

The log-canonical parameters for the Poisson variety $\mathcal{A}_n$ were built in \cite{10.1093/imrn/rnac101} by modifying the Fock--Goncharov construction.
The quiver of the corresponding Poisson structure is obtained by amalgamating and removing certain vertices from the quiver describing Fock--Goncharov parameters for $n=5$, similar to one on Figure~\ref{fig:FG}, as explained in Section~\ref{sec:logcanonicalgroupoid}. The result is the quiver on Figure~\ref{quiverA5} for $\mathcal{A}_5$ (a corresponding quiver for~$\mathcal{A}_6$ is shown on Figure~\ref{fig:amalg}).

\begin{figure}[t]
\centering
\includegraphics[width=40mm]{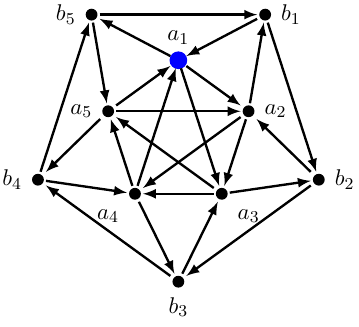}

\caption{Exchange quiver of the cluster structure for~$\mathcal{A}_5$ describing the Poisson bracket in log-canonical coordinates.}\label{quiverA5}
\end{figure}

\subsubsection{Coisotropic reduction in case~(1)}

We now describe the procedure that relates the quiver in Figure~\ref{quiverA5} with the quiver in Figure~\ref{quiverTgs} under the condition that $C_1=-1$. Note that the number of vertices of the first quiver is~10, while the number of vertices of the second quiver is~9.
This indicates that this procedure must include a coisotropic reduction.

Let us introduce the following notation.
For any $z_1,z_2,\dots,z_n$, let
\begin{gather*}
 \langle z_{1}z_{2} \cdots z_{n-1}z_{n} \rangle := (z_{1}z_{2} \cdots z_{n-1}z_{n} )^{\frac{1}{2}}\\
 \hphantom{\langle z_{1}z_{2} \cdots z_{n-1}z_{n} \rangle :=}{}\times
\left(1+\frac{1}{z_{1}}+\frac{1}{z_{1}z_{2}}+ \cdots+\frac{1}{z_{1}z_{2} \cdots z_{n-1}}+\frac{1}{z_{1}z_{2} \cdots z_{n-1}z_{n}}\right).
\end{gather*}

Let $\hat A_5$ be the unipotent $5\times 5$ matrix constructed from the original $\mathcal A_5$-quiver in Figure~\ref{quiverA5}, and let $a_{ij}$ be its $ij$-th entry.
As discussed at the end of Section~\ref{sec:logcanonicalgroupoid}, we obtain uniform expressions:
\begin{align}
 a_{1,2}=\bigl\langle b_1a_1a_3b_2\bigr\rangle,\qquad \!\!
a_{2,3}=\bigl\langle b_2a_2a_4b_3\bigr\rangle,\qquad\!\!
  a_{3,4}=\bigl\langle b_3a_3a_5b_4\bigr\rangle,
\qquad\!\!  a_{4,5}=\bigl\langle b_4a_4a_1b_5\bigr\rangle.\!\!\!  \label{Gij-5}
\end{align}

\begin{Remark}
Note that in both the $\mathcal A_n$ and geometric cases, elements $a_{i,i+1}(=G_{i,i+1})$ with $1\le i<n$ generate the full set of matrix elements using Poisson and skein relations: for any $i<j<k$, $a_{i,k}=\{a_{i,j},a_{j,k}\}+\frac12 a_{i,j}a_{j,k}$, so, starting with $a_{i,i+1}$, we can consecutively construct all other $a_{i,k}$.
\end{Remark}

\begin{Remark}
The expressions of all $a_{i,j}$ are universal Laurent polynomials; i.e., they retain their polynomial forms upon all sequences of mutations. This is true for all matrix elements in all $\mathcal A_n$ quivers. For a general $\mathcal A_n$-quiver having a form similar to that in Figures~\ref{quiverA5} and~\ref{fig:clusterA6}, the generating matrix elements $a_{i,i+1}$ are given by
$a_{i,i+1}=\bigl\langle x_i y_{i_1} z_{j_1} w_{k_1} \cdots w_{k_2} z_{j_2} y_{i_2} x_{i+1}\bigr\rangle$, where $x_i$ and $x_{i+1}$ are cluster variables of the ``outer'' layer of the $\mathcal A_n$ quiver (the corresponding vertices are of order four) connected by an arrow $x_i\to x_{i+1}$. All cluster variables in this sequence form a directed path in $\mathcal A_n$ quiver:
\[
x_i\to y_{i_1} \to z_{j_1} \to w_{k_1} \to \cdots \to w_{k_2} \to z_{j_2} \to y_{i_2} \to x_{i+1},
\]
and the path is the shortest path that passes through all layers of the quiver and includes exactly two variables of each layer except the innermost layer for which it includes two variables for odd $n$ and one variable for even $n$.
\end{Remark}

Note that in the cases $n=3$ and $n=4$, the image ${\rm Im}(\varphi)$ is an open subset in $\mathcal{A}_n$. One can check that, in this case, the geodesic functions can be written uniquely in terms of (square roots of) log-canonical parameters. The relation to the corresponding cluster variables of $\mathcal{A}_3$ and $\mathcal{A}_4$ was presented in~\cite{10.1093/imrn/rnac101}.

In case~(1) for $n=5$, we proceed by performing a cluster mutation $\mu_{a_1}$ at the large highlighted blue vertex $a_1$, followed by mutations $\mu_{a_4}$ at $a_4$ and $\mu_{a_3}$ at $a_3$, which yields the following quiver (see Figure~\ref{quiverA5modified}).

\begin{figure}[ht]
\centering
\includegraphics[width=40mm]{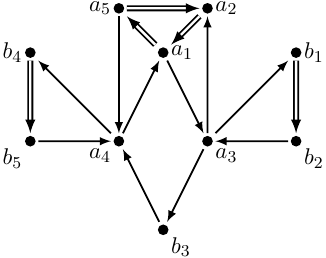}

\caption{Quiver for $\mathcal{A}_5$ after cluster mutation.}\label{quiverA5modified}

\end{figure}

In the next step of the transformation, we make a coisotropic reduction.
To do it, we compute the geodesic functions $G_{i,j}$ for $\mathcal{T}_{2,1}$ and compare them with the corresponding generalized geodesic functions $a_{i,j}$ in $\mathcal{A}_5$.
We express geodesic functions $G_{ij}$ in terms of shear coordinates of the ribbon graph in Figure~\ref{fatgraph} (or, equivalently, in terms of quiver variables in Figure~\ref{quiverTgs}). We collect them in the right-hand side of~(\ref{eq:geodesicN5}).

Rewriting expressions (\ref{Gij-5}) after the mutations $\mu_{a_4}\mu_{a_3}\mu_{a_1}$ (the order of mutations is right to left)
of the cluster in Figure~\ref{quiverA5}, we obtain the expressions for $a_{i,i+1}$ in terms of $a_i$, $b_i$ parameters of quiver in Figure~\ref{quiverA5modified} and collect them in the left-hand side of the following system:
\begin{align}
 & a_{1,2} = \langle b_{1}b_{2} \rangle = \langle v_{1}v_{2} \rangle=G_{1,2}, \nonumber\\
 & a_{2,3} = \langle b_{2}a_{3}a_{2}a_{1}a_{3}b_{3} \rangle \nonumber\\
 &\hphantom{a_{2,3}}{} = (v_{2}h_{4}h_{1}v_{3} )^{\frac{1}{2}}\left(1+\frac{1}{v_{2}}+\frac{1}{h_{1}v_{2}}+\frac{1}{h_{4}v_{2}}+\frac{1}{h_{4}h_{1}v_{2}}+\frac{1}{v_{3}h_{4}h_{1}v_{2}}\right) =G_{2,3},\nonumber\\
 & a_{3,4} = \langle b_{3}a_{4}a_{1}a_{5}a_{4}b_{4} \rangle \nonumber\\
 & \hphantom{a_{3,4}}{} = (v_{3}h_{2}h_{3}v_{4} )^{\frac{1}{2}}\left(1+\frac{1}{v_{3}}+\frac{1}{v_{3}h_{3}} +\frac{1}{v_{3}h_{2}}+\frac{1}{v_{3}h_{2}h_{3}}+\frac{1}{v_{3}h_{2}h_{3}v_{4}}\right)=G_{3,4},\nonumber\\
 & a_{4,5} = \langle b_{4}b_{5} \rangle = \langle v_{4}v_{5} \rangle=G_{4,5}, \qquad
 C_1=a_1a_2a_5,\qquad C_2=\prod_{i=1}^5 b_i a_3^2a_4^2a_1. \label{eq:geodesicN5}
\end{align}

We wish to find expressions for $v_i$, $h_i$ in terms of $a_j$, $b_j$ that make~(\ref{eq:geodesicN5}) hold (a solution is clearly not unique: although, in $n=5$ case, the number of variables is the same if we take into account the constraint $a_1a_2a_5=-1$, there is an infinite braid-group symmetry that preserves the quiver).

Inspection suggests the substitution
\begin{equation}\label{subsx}
v_i=b_i \qquad \text{for all } i \in \{1,2,3,4,5\}.
\end{equation}
Performing this substitution simplifies the equation for $G_{2,3}$ to
\[ \bigl(a_1a_2a_3^2\bigr)^{\frac{1}{2}}\left(\frac{1}{a_3}+\frac{1}{a_2a_3}+\frac{1}{a_1a_2a_3}+\frac{1}{a_1a_2a_3^2}\right) = (h_{4}h_{1})^{\frac{1}{2}}\left(\frac{1}{h_{1}}+\frac{1}{h_{4}}+\frac{1}{h_{1}h_{4}}\right). \]
This equation holds if the following two conditions are satisfied:
\begin{enumerate}\itemsep=0pt
 \item[(1)] $\bigl(a_1a_2a_3^2\bigr)^\frac{1}{2} = (h_4h_1)^\frac{1}{2} \implies h_1h_4=a_1a_2a_3^2$,
 \item[(2)] $\dfrac{1}{a_3}+\dfrac{1}{a_2a_3}+\dfrac{1}{a_1a_2a_3}+\dfrac{1}{a_1a_2a_3^2} = \dfrac{1}{h_{1}}+\dfrac{1}{h_{4}}+\dfrac{1}{h_{1}h_{4}} \implies h_1+h_4=a_1a_2a_3+a_1a_3+a_3$.
\end{enumerate}
Condition $(2)$ implies the relation $h_4 = a_1a_2a_3+a_1a_3+a_3 - h_1$ whose substitution into condition~$(1)$ yields the equation
\begin{align*}
 h_1(a_1a_2a_3+a_1a_3+a_3 - h_1) = a_1a_2a_3^2,
\end{align*}
which becomes the quadratic equation
\begin{align*}
 h_1^2 - (a_1a_2a_3+a_1a_3+a_3)h_1 + a_1a_2a_3^2=0.
\end{align*}
Solving for $h_1$ and observing that the conditions are symmetric for $h_1$ and $h_4$ yields
\begin{equation}
\label{subsy14}
h_1, h_4 = \frac{a_3}{2}\Bigl(a_1a_2+a_1+1 \pm \sqrt{a_1^2a_2^2+2a_1^2a_2-2a_1a_2+a_1^2+2a_1+1}\Bigr).
\end{equation}
Similar reasoning for $G_{34}$ implies
\begin{equation}
\label{subsy23}
h_2, h_3 = \frac{a_4}{2}\Bigl(a_1a_5+a_5+1 \pm \sqrt{a_1^2a_5^2+2a_1a_5^2-2a_1a_5+a_5^2+2a_5+1}\Bigr).
\end{equation}

We claim now that the variables $v_1,\dots,v_5,h_1,\dots,h_4$ given by formulas (\ref{subsx}),~(\ref{subsy14}),~(\ref{subsy23})
satisfy Poisson commutation relations determined by the quiver in Figure~\ref{quiverTgs} under the condition that the Casimir function $C_1=a_1 a_2 a_5$ of the Poisson bracket defined by the quiver from Figure~\ref{quiverA5modified} takes the value $-1$.

Indeed, clearly, $a_1 a_2 a_5$ is a Casimir function of the quiver in Figure~\ref{quiverA5modified}.
We are interested in the case where the value of the Casimir function
\begin{equation}\label{eq:CasA5}
 a_1a_2a_5=-1.
\end{equation}
This critical condition is equivalent to rank condition (Corollary~\ref{cor:rank}), the only one we need to make the reduction.
To verify the expressions of the $v_i$ and $h_j$ in terms of the $b_i$ and $a_j$, we seek to verify the commutation relations of the quiver in Figure~\ref{quiverTgs}.
Note that by construction, $\{v_i,v_j\}=\{b_i,b_j\}$.
Note also that all $b_i$ Poisson commute with any variable of $\{a_1,a_2,a_5\}$. Comparing the quivers in Figures~\ref{quiverTgs} and \ref{quiverA5modified}, we immediately conclude that the arrows of the quiver in Figure~\ref{quiverTgs} between $v_i$ and $v_j$ and between $v_i$ and $h_j$ reflect correct Poisson relations between corresponding expressions~(\ref{subsx}), (\ref{subsy14}), (\ref{subsy23}).

It remains to check the Poisson commutation relations between $h_1$, $h_2$, $h_3$, $h_4$. We show first that $\{h_1,h_4\}=0$. Indeed,
$\{h_1+h_4,h_1h_4\}=(h_1-h_4)\{h_1,h_4\}$. It remains to notice that $h_1h_4=a_1 a_2 a_3^2$, which Poisson-commutes with $a_1$, $a_2$ and $a_3$, hence, with any function of~$a_1$,~$a_2$,~$a_3$. Therefore,
$(h_1-h_4)\{h_1,h_4\}=0$, because $h_1$ and $h_4$ are functions of $a_1$, $a_2$ and $a_3$ only. Since $h_1-h_4\ne 0$, we conclude that $\{h_1,h_4\}=0$. Similarly, we prove that $\{h_2,h_3\}=0$.

Direct computation shows $\{h_3,h_4\}=h_3 h_4$ and, similarly, $\{h_2,h_1\}=h_2 h_1$.

Next, we verify that $\{ h_2+h_3, h_1+h_4 \} = \{ h_2, h_1 \}+\{ h_3, h_4 \}$. Recall that $a_5=-\frac{1}{a_1a_2}$ to compute
\begin{align*}
&\{h_2,h_1\}+\{h_3,h_4\}= \frac{a_3a_4}{2} \Bigl((a_1a_5+a_5+1)(a_1a_2+a_1+1) \\
& \hphantom{\{h_2,h_1\}+\{h_3,h_4\}=}{}
 \pm \sqrt{((a_1a_5+a_5+1)^2-4a_1a_5)\bigl((a_2a_1+a_1+1)^2-4a_1a_2\bigr)} \Bigr)\\
& \hphantom{\{h_2,h_1\}+\{h_3,h_4\}}{}=\frac{a_3a_4}{2} \Bigl( \Bigl(-\frac{1}{a_2}-\frac{1}{a_1a_2}+1\Bigr)(a_1a_2+a_1+1)\\
& \hphantom{\{h_2,h_1\}+\{h_3,h_4\}=}{} \pm\frac{1}{a_1a_2}\bigl((a_1a_2+a_1+1)^2-4a_1a_2\bigr) \Bigr)\\
& \hphantom{\{h_2,h_1\}+\{h_3,h_4\}}{} =\frac{a_3a_4}{2a_1a_2} \Bigl( \bigl(a_1^2a_2^2-a_1^2-2a_1-1\bigr)\\
& \hphantom{\{h_2,h_1\}+\{h_3,h_4\}=}{} +\bigl(a_1^2a_2^2+2a_1^2a_2+2a_1a_2+2a_1+a_1^2+1-4a_1a_2\bigr) \Bigr)\\
& \hphantom{\{h_2,h_1\}+\{h_3,h_4\}}{} =\frac{a_3a_4}{2a_1a_2} \Bigl( 2a_1^2a_2^2+2a_1^2a_2-2a_1a_2 \Bigr)\\
& \hphantom{\{h_2,h_1\}+\{h_3,h_4\}}{} = a_3a_4(a_1a_2+a_1-1),
\\
&\{h_2+h_3,h_1+h_4\} =\{a_4(a_1a_5+a_5+1), a_3(a_1a_2+a_1+1)\}\\
& \hphantom{\{h_2+h_3,h_1+h_4\}}{}
=\Bigl\{a_4\Bigl(-\frac{1}{a_2}-\frac{1}{a_1a_2}+1\Bigr), a_3(a_1a_2+a_1+1)\Bigr\}\\
& \hphantom{\{h_2+h_3,h_1+h_4\}}{}
=a_4(a_1a_2+a_1+1)\Bigl\{-\frac{1}{a_2}-\frac{1}{a_1a_2}+1, a_3\Bigr\}\\
& \hphantom{\{h_2+h_3,h_1+h_4\}=}{} +a_3\Bigl(-\frac{1}{a_2}-\frac{1}{a_1a_2}+1\Bigr)\{a_4, a_1a_2+a_1+1\}\\
& \hphantom{\{h_2+h_3,h_1+h_4\}=}{} +a_3a_4\Bigl\{-\frac{1}{a_2}-\frac{1}{a_1a_2}+1, a_1a_2+a_1+1\Bigr\}\\
& \hphantom{\{h_2+h_3,h_1+h_4\}}{} =a_4(a_1a_2+a_1+1)\Bigl(-\frac{a_3}{a_2}\Bigr)
 +a_3\Bigl(-\frac{1}{a_2}-\frac{1}{a_1a_2}+1\Bigr)(a_1a_2a_4+a_1a_4)\\
& \hphantom{\{h_2+h_3,h_1+h_4\}=}{} +a_3a_4\Bigl(2a_1+\frac{2a_1+2}{a_2}\Bigr)\\
& \hphantom{\{h_2+h_3,h_1+h_4\}}{} =a_3a_4\Bigl(-a_1-\frac{a_1}{a_2}-\frac{1}{a_2}-a_1-\frac{a_1}{a_2}-1-\frac{1}{a_2}\\
& \hphantom{\{h_2+h_3,h_1+h_4\}=}{} +a_1a_2+a_1+2a_1+\frac{2a_1}{a_2}+\frac{2}{a_2}\Bigr)\\
& \hphantom{\{h_2+h_3,h_1+h_4\}}{}=a_3a_4(a_1a_2+a_1-1).
\end{align*}
Subtracting the first equation from the second, we obtain
\begin{equation}\label{y24andy31}
 \{h_2,h_4\}+\{h_3,h_1\}=0.
\end{equation}
Note that symmetric polynomials of $h_1$, $h_4$ and $h_2$, $h_3$ are expressible as polynomials in~$a_1$,~$a_2$,~$a_3$ and $a_1,a_4,a_5$, correspondingly. From such expressions of $h_1+h_4$ and $h_2 h_3$ and commutation relations for~$a_i$, we easily see that $\{h_1+h_4,h_2 h_3\}=(h_1+h_4)h_2 h_3$.
In particular,
\begin{equation}\label{y13Andy42}
 \{h_1,h_3\}h_2+\{h_4,h_2\}h_3=0.
\end{equation}

Combining~\eqref{y24andy31} and~\eqref{y13Andy42} and taking into account that the determinant
\[
\det \begin{pmatrix}
 -1 & -1 \\
 h_2 & h_3
\end{pmatrix}=h_2-h_3\ne 0
\] for generic $h_2$, $h_3$,
we conclude $\{h_2,h_4\}=0$ and $\{h_1,h_3\}=0$.

Therefore, under the rank condition
$a_1 a_2 a_5=-1$ the Poisson bracket between the variables $v_1$, $v_2$, $v_3$, $v_4$, $v_5$, $h_1$, $h_2$, $h_3$, $h_4$ given by \eqref{subsx}, \eqref{subsy14}, \eqref{subsy23} is described by the quiver in Figure~\ref{quiverTgs}.

We have found an explicit form of the Poisson map $\mathcal{T}_{2,1}\rightarrow \mathcal{A}_5$ in terms of coordinates $v_i,h_j$ in Case (1).

\begin{Remark}
Note that, unlike higher $n$ cases, for $n=5$ we can invert formulas (\ref{subsy14}) and (\ref{subsy23}) taking condition (\ref{eq:CasA5}) into account and express $a_i$, $i=1,\dots,5$, in terms of $h_j$, $j=1,\dots,4$ (see~(\ref{eq:AX})). The resulting $a_i(h_j)$ are rational functions, and it is straightforward to check that the homogeneous Poisson relations for variables $h_j$ induce those for $a_i(h_j)$.
\end{Remark}

\subsubsection{Case (2)}

Instead of performing reduction directly in this case, we present a procedure that reduces this case to case~(1) via invertible sequences of mutations. For this, let us introduce a convenient notation, which we also use in the $n=6$ case.

In any quiver, monomial Casimirs $x_1^{\alpha_1}x_2^{\alpha_2}\cdots x_n^{\alpha_n}$ with integers $\alpha_i$ are determined by null vectors of the exchange matrix~$B$: $B\vec\alpha=\vec 0$. Upon mutation at $x_j$, $\vec \alpha\to {\vec \alpha}{}'$ with
\begin{equation*}
\alpha'_i=\alpha_i \qquad \text{for } i\ne j \qquad
\text{and} \qquad
\alpha'_j= -\alpha_j+\sum_{i}B^{(+)}_{i,j}\alpha_i,
\end{equation*}
 where the sum ranges over only the positive entries of the matrix $B$. Note that the nullity condition implies that
\[
\sum_{i}B^{(+)}_{i,j}\alpha_i + \sum_{i}B^{(-)}_{i,j}\alpha_i=0\qquad \forall j.
\]
Here, $B^{(+)}_{ij}=\max(B_{ij},0)$, $B^{(-)}_{ij}=\min(0,B_{ij})$.
We leave it as a simple but nice exercise for the reader to check that this condition is preserved under mutations.

We start with the quiver in Figure~\ref{quiverA5} assigning to each vertex a 2-row vector in which the first component is the power of this variable in the expression for $C_1$ and the second component is the power of this variable in the expression for $C_1C_2$. In Figure~\ref{quiverA5muted}, we show the original quiver and the quiver obtained from it by successive mutations at $b_2$, $b_4$, and $b_3$.
Note that the Casimirs $C_i$ preserve monomial form but the powers of mutated cluster variables in the monomial expressions for Casimirs may change because the variables themselves change.

\begin{figure}[ht]
\centering
\includegraphics[width=140mm]{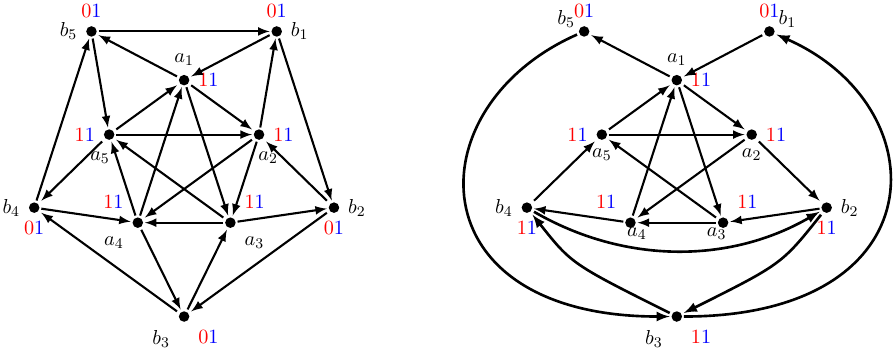}

\caption{Left: $\mathcal{A}_5$-quiver with powers of Casimirs $C_1$ and $C_2$ indicated. Right: the same quiver after mutations at $b_4$, $b_2$, and $b_3$.}\label{quiverA5muted}

\end{figure}

The only difference between powers of $C_1$ and $C_1 C_2$ in the initial quiver (on the left in Figure~\ref{quiverA5muted}) and in the transformed quiver (on the right) appears at vertices $b_1$ and $b_5$. If we now perform mutations at these two vertices and permute variables~$b_1$ and~$b_5$, we obtain the quiver (see Figure~\ref{Figure23a}),
which has exactly the same form, except that~$b_1$ switches positions with $b_5$ and all powers of variables in $C_1$ and $C_1C_2$ are interchanged. Performing mutations at $b_3$, $b_2$, and~$b_4$ then returns us to the original ${\mathcal A}_5$-quiver with values of $C_1$ and $C_1C_2$ interchanged. We can then perform reduction as described in case~(1).

\begin{figure}[t]
\centering
\includegraphics[width=65mm]{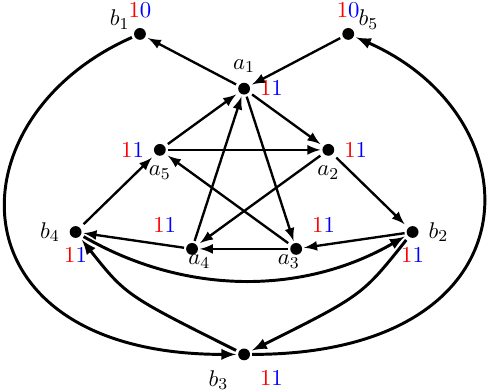}

\caption{Positions of $b_1$ and $b_5$, and exponents of variables in $C_1$ and $C_1 C_2$ are interchanged.}\label{Figure23a}

\end{figure}

\begin{Remark}
Despite the fact that for the coisotropic reduction in this section we used the rank condition on the matrix~$A$ whose elements $G_{ij}$ are $\mathbb{Z}_2$ invariant, the cluster algebra obtained as the result of the reduction contains, in particular, all exponential shear coordinates as cluster variables and, hence, all geodesic functions (symmetric or not) are universal Laurent polynomials of the square roots of these cluster variables.
\end{Remark}

\subsection[n=6]{$\boldsymbol{n = 6}$}

\subsubsection[The rank condition for n=6 and Fock--Goncharov parameters]{The rank condition for $\boldsymbol{n = 6}$ and Fock--Goncharov parameters}

The rank condition for $n=6$ claims that ${\rm corank}\bigl(A+A^{\mathsf{T}}\bigr)=2$.
We now use the fact that the rank condition can be rewritten~(\ref{eq:rank-AM}) in terms of the matrix~$\mathcal{M}_3 S$.
Let us denote the elements of the main diagonal of~$\mathcal{M}_3$ by $m_{i,i}$ (where~$i$ is the row index). Recall that $S$ is anti-symmetric for $n=6$. Then both $\mathcal{M}_3 S$ and $(\mathcal{M}_3 S)^{\mathsf{T}}$ are lower anti-triangular with entry
$(-1)^{i-1}m_{ii}$ in the $i$th row of the anti-diagonal of~$\mathcal{M}_3 S$
and $(-1)^{i}m_{n-i+1,n-i+1}$ in the $i$th row of the anti-diagonal of
 $S^{\mathsf{T}} \mathcal M_3^{\mathsf{T}} = (\mathcal M_3 S)^{\mathsf{T}}$.

Since the matrix $\mathcal{M}_3 S+S^{\mathsf{T}}\mathcal{M}_3^{\mathsf{T}}$ is symmetric, it contains three pairs of equal anti-diagonal elements $m_{11}-m_{61}$, $-m_{25}+m_{52}$, $m_{34}-m_{43}$.
The equation $\det\bigl(\mathcal M_3S + S^{\mathsf{T}}\mathcal M_3^{\mathsf{T}}\bigr) = 0$ is equivalent to the condition (called \emph{Casimir condition}) that at least one of these three anti-diagonal elements vanishes.
Introduce the anti-diagonal lower-triangular $6\times 6$ matrix $\Psi=(\psi_{ij})_{i,j=1}^6$.

By the reasoning above, the Rank Condition $\text{rank}\bigl(A_6+A_6^{\mathsf{T}}\bigr)\leq 4$ for $n=6$
means that $\operatorname{corank}(\Psi)\ge 2$. Corank is at least one when two symmetric elements of the anti-diagonal vanish. To increase the corank by one more, we must require the vanishing of an additional minor. More exactly, the rank condition splits into the following three cases:
\begin{enumerate}\itemsep=0pt
 \item\label{case1}
\emph{Casimir condition} $\psi_{34}=\psi_{43}=0$ and $\psi_{44}=0$.
 \item\label{case2}
\emph{Casimir condition} $\psi_{25}=\psi_{52}=0$ and $\det\bigl(\Psi_{[3,5]}^{[3,5]}\bigr)=0$.

\item\label{case3}
\emph{Casimir condition} $\psi_{16}=\psi_{61}=0$ and
 $\det\bigl(\Psi_{[2,6]}^{[2,6]}\bigr)=0$.
\end{enumerate}

Rewriting all three cases in terms of variables $a_i$, $b_j$, $c_k$ we obtain the following:
\begin{description}\itemsep=0pt
\item[Case \ref{case1}:] The Casimir condition is \smash{$K_1= \prod_{j=1}^3 c_j =1$}. The condition $\psi_{4,4}=0$ becomes $1+c_2+c_2 c_3 = 0$.
\item[Case \ref{case2}:] The Casimir condition is \smash{$K_2= \prod_{i=1}^6 b_i \prod_{j=1}^3 c_j =1$}. The condition \smash{$\det\bigl(\Psi_{[3,5]}^{[3,5]}\bigr)=0$} is a~polynomial containing $108$ monomial terms,
 {all with positive signs}, which we omit here.
\item[Case~\ref{case3}:] The Casimir condition is \smash{$K_3= \prod_{i=1}^6 a_i b_i \prod_{j=1}^3 c_j =1$}.
The expansion of the function \smash{$\det\bigl(\Psi_{[2,6]}^{[2,6]}\bigr)$} contains 3,888 monomials (counted with multiplicities), again, all with positive signs which we omit here.
\end{description}

In the following subsections, we analyze these cases more carefully, beginning with the simplest case, Case~\ref{case1}.

\subsubsection{Coisotropic reduction: Case~\ref{case1}}

 We describe the coisotropic reduction that restores exponential shear coordinates and the Goldman Poisson structure on the geometric leaf in Case~\ref{case1} of the solution of the rank condition in~$\mathcal{A}_6$.

Recall that the amalgamation procedure described in \cite{10.1093/imrn/rnac101} and Section~\ref{sec:logcanonicalgroupoid}
replaces the product $Z_{\ell, 0, 6-\ell}Z_{6-\ell, \ell, 0}$ by one variable because the expression for any geodesic function contains solely the product $Z_{\ell, 0, 6-\ell}Z_{6-\ell, \ell, 0}$ for all $\ell=1,\dots,5$. In particular, the amalgamated coordinates~$a_1$,~$a_2$, $a_3$, $a_4$, $a_5$, $a_6$, $b_1$, $b_2$, $b_3$, $b_4$, $b_5$, $b_6$, $c_1$, $c_2$, $c_3$ were introduced in~\cite{10.1093/imrn/rnac101}. Each $a_i$, $b_i$, or $c_i$ is either one of the parameters $Z_{ijk}$ or the product $Z_{\ell, 0, 6-\ell}Z_{6-\ell, \ell, 0}$ such that all entries of $A\in \mathcal{A}_6$ are Laurent polynomial expressions in \smash{$a_i^{1/2}$}, \smash{$b_j^{1/2}$}, \smash{$c_k^{1/2}$}. The Poisson bracket for~$a_i$,~$b_j$,~$c_k$ is described by the quiver in Figure~\ref{fig:clusterA6} (isomorphic to the quiver in Figure~\ref{fig:amalg}).

\begin{figure}[ht]
\centering
\includegraphics[width=50mm]{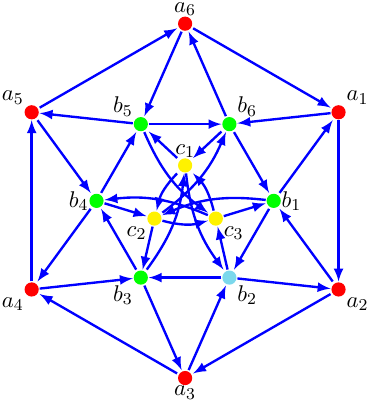}

\caption{Quiver of the Poisson bracket and cluster structure for $\mathcal{A}_6$.} \label{fig:clusterA6}

\end{figure}

The elementary geodesic functions $G_{i,i+1}$, $i=1,\dots,5$, and $G_{6,1}$ are
\begin{align*}
&\langle a_6b_5c_3b_1a_1 \rangle,\qquad \langle a_1b_6c_1b_2a_2 \rangle, \qquad \langle a_2b_1c_2b_3a_3 \rangle, \\
& \langle a_3b_2c_3b_4a_4 \rangle,\qquad \langle a_4b_3c_1b_5a_5 \rangle,\qquad \langle a_5b_4c_2b_6a_6 \rangle.
\end{align*}
The reduction conditions are
\[
c_1c_2c_3=1\qquad\text{and}\qquad 1+\frac1{c_2}+\frac{1}{c_2c_3}=0.
\]
Mutating at $c_1$, we obtain the transformed $n=6$ quiver
depicted in Figure~\ref{fig:clusterA6mutated}.

\begin{figure}[t]
\centering
\includegraphics[width=50mm]{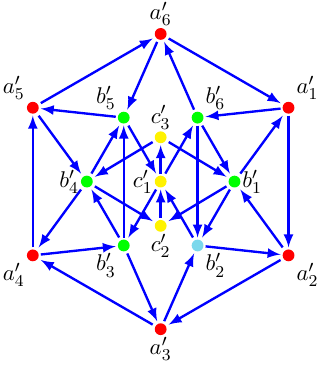}

\caption{$\mathcal{A}_6$-quiver upon mutation at $c_1$.} \label{fig:clusterA6mutated}
\end{figure}

In the quiver in Figure~\ref{fig:clusterA6mutated}, the corresponding elementary geodesic functions are
\begin{align*}
&\langle a'_6b'_5c'_1c'_3b'_1a'_1 \rangle,\qquad \langle a'_1b'_6b'_2a'_2 \rangle, \qquad \langle a'_2b'_1c'_2c'_1b'_3a'_3 \rangle, \\
& \langle a'_3b'_2c'_1c'_3b'_4a'_4 \rangle,\qquad \langle a'_4b'_3b'_5a'_5 \rangle,\qquad \langle a'_5b'_4c'_2c'_1b'_6a'_6 \rangle.
\end{align*}
By prime, we denote the corresponding cluster variables after mutations.

Let us rewrite the reduction conditions in the new cluster variables:
\begin{align*}
&c_1c_2 c_3=1\iff\dfrac{1}{c'_1}c'_2\dfrac{c'_1}{1+c'_1}c'_3(1+c'_1)=c'_2c'_3=1,\\
&1+\dfrac{1}{c_2}+\dfrac{1}{c_2c_3}=0\iff 1+\dfrac{1+c'_1}{c'_2c'_1}+\frac{1+c'_1}{c'_1c'_2c'_3(1+c'_1)}\\
&\qquad = 1+\dfrac{1+c'_1}{c'_2c'_1}+\frac{1}{c'_1c'_2c'_3} \overset{c'_2c'_3=1}{=\joinrel=\joinrel=} 1+\dfrac{1}{c'_1c'_2}+\dfrac{1}{c'_2}+\frac{1}{c'_1}=0
 \iff\left(1+\dfrac{1}{c'_2}\right)\left(1+\dfrac{1}{c'_1}\right)=0.
\end{align*}

We conclude that the reduction conditions become
\[
c'_2c'_3=1,\qquad \left(1+\frac1{c'_1}\right)\left(1+\frac1{c'_2}\right)=0,
\]
whose solution is
\begin{equation}\label{cc}
c'_2=-1,\qquad c'_3=-1.
\end{equation}
Note that the solution $c_1'=-1$ is not in the set of allowed values for mutation at $c_1$.

Consider the expression $\langle a'_6b'_5c'_1c'_3b'_1a'_1 \rangle$.
Imposing conditions~(\ref{cc}) and introducing new variables $\widetilde b_1=-c'_1 b'_1$, $\widetilde b_4=-c'_1 b'_4$,
we can write all six geodesic functions as follows:
$\langle a'_6b'_5\widetilde b_1a'_1 \rangle$, $\langle a'_1b'_6b'_2a'_2 \rangle$, $\langle a'_2\widetilde b'_1b'_3a'_3 \rangle$,
 $\langle a'_3b'_2\widetilde b_4a'_4 \rangle$, $\langle a'_4b'_3b'_5a'_5 \rangle$, $\langle a'_5\widetilde b_4b'_6a'_6 \rangle$.

Amalgamating $-b'_1c'_1=\widetilde b_1$ and $-b'_4c'_1=\widetilde b_4$, we obtain a new quiver, in which $c'_2$ and $c'_3$ are detached because they Poisson-commute with all the variables obtained after amalgamation.

\begin{figure}[t]
\centering
\includegraphics[width=50mm]{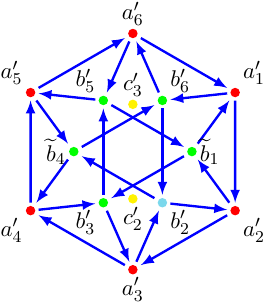}

\caption{Quiver of the Poisson bracket for variables $a_1'$, $a_2'$, $a_3'$, $a_4'$, $a_5'$, $a_6'$, ${\widetilde b_1}$, $b_2'$, $b_3'$, ${\widetilde b_4}$, $b'_5$, $b'_6$.} \label{fig:reducedquiver}
\end{figure}

Note that all vertices of the quiver (Figure~\ref{fig:reducedquiver}) are four-valent with two incoming arrows
and two outgoing arrows. It is straightforward to check that this quiver corresponds to a ribbon graph (Figure~\ref{fig:fat_graph}) for $g=2$, $s=2$ with vertices of the quiver corresponding to edges of the ribbon graph.

\begin{figure}[t]
\centering
\includegraphics[width=50mm]{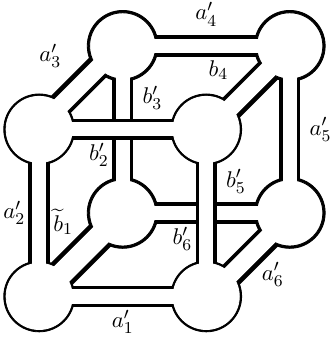}
\vspace{-1mm}

\caption{Ribbon graph of a genus-$2$ surface with $2$ holes.} \label{fig:fat_graph}
\end{figure}

The geodesic functions take the form
\begin{align*}
\begin{split}
&{\rm tr} (R_{a'_1} L_{\widetilde b_1} L_{b'_5}L_{a'_6} ),\qquad {\rm tr} (R_{a'_2} L_{b'_2} L_{b'_6}L_{a'_1} ), \qquad {\rm tr} (R_{a'_3} L_{b'_3} L_{\widetilde b_1} L_{a'_2} ),\\
&{\rm tr} (R_{a'_4} L_{\widetilde b_4} L_{b'_2}L_{a'_3} ),\qquad {\rm tr} (R_{a'_5} L_{b'_5} L_{b'_3}L_{a'_4} ),\qquad {\rm tr} (R_{a'_6} L_{b'_6} L_{\widetilde b_4} L_{a'_5} ).
\end{split}
\end{align*}

\begin{Remark}
 Note that condition~\eqref{cc} indicates that $c'_2$ and $c_3'$ do not belong to {
 the set of allowed values} for mutations at $c_2'$ and $c_3'$.
\end{Remark}

\subsubsection{Coisotropic reduction: Cases~\ref{case2} and~\ref{case3}}

Instead of trying to construct a straightforward coisotropic reduction for Cases~\ref{case2} and~\ref{case3}, we use an indirect approach
which we already used in the case $n=5$. More precisely, we find sequences of mutations that transform Cases~\ref{case2} and~\ref{case3} to Case~\ref{case1}.
Note that the cases differ in which elements of $\Psi$ vanish. The vanishing of the corresponding elements corresponds to setting the value of the corresponding Casimir function $K_1$, $K_2$, or $K_3$ equal to $1$ as explained above. Each Casimir function is a monomial in the variables $a_i$, $b_i$, and $c_i$ that are assigned to the vertices of the quiver in Figure~\ref{fig:clusterA6}. We assign to every variable a 3-row vector in~${\mathbb Z}^3$, whose components are the exponents of the corresponding variable in the respective Casimir functions. In the figures below, the red color corresponds to exponents in $K_1$, the green to $K_2$, and the blue to $K_3$.
In the following figure, we see that $ K_1 = c_1c_2c_3$, and so on.

We observe first that the mutation sequences below permute the Casimir functions of the quiver in Figure~\ref{fig:clusterA6}. Then we check that the sequences permute the corresponding minors $\psi_{4,4}$, \smash{$\det\bigl(\Psi_{[3,5]}^{[3,5]}\bigr)$}, and \smash{$\det\bigl(\Psi_{[2,6]}^{[2,6]}\bigr)$}.
Beginning with the quiver for $\mathcal{A}_6$, we perform the following mutations at the bold highlighted vertices. Under the mutations, $K_1$, $K_2$, and $K_3$ remain monomials of mutated variables whose exponents are the components of vectors written in Figure~\ref{Figure-27a}.

\begin{figure}[ht]
\centering
\includegraphics[width=49mm]{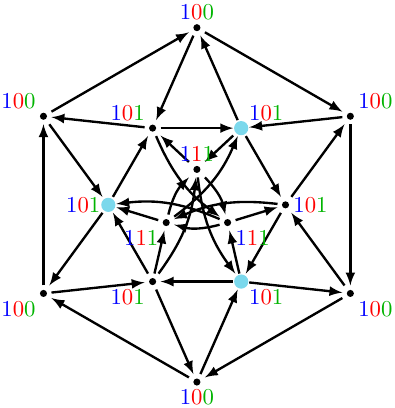}
\vspace{-2mm}

\caption{Exponents of variables in expressions of $K_1$, $K_2$, and $K_3$ are written as vectors with colored entries.} \label{Figure-27a}\vspace{-1mm}
\end{figure}

The result is shown in Figure~\ref{Figure-27b}. We continue with the sequence of mutations at the bold highlighted vertices.

\begin{figure}[!ht]
\centering
\includegraphics[width=50mm]{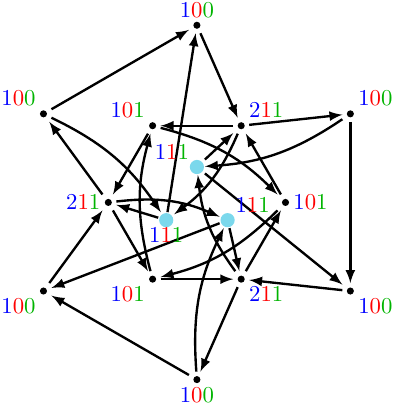}
\vspace{-1mm}

\caption{Result of three commuting mutations at  the fat lightblue vertices of the quiver in Figure~\ref{Figure-27a}.} \label{Figure-27b}
\end{figure}

This yields the following quiver depicted in Figure~\ref{Figure-27c}.

\begin{figure}[ht]
\centering
\includegraphics[width=50mm]{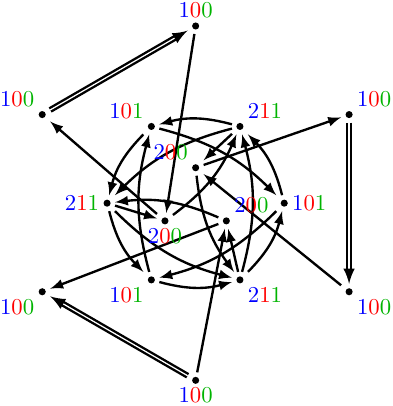}
\vspace{-1mm}

\caption{Result of three commuting mutations at the fat lightblue vertices of the quiver in Figure~\ref{Figure-27b}.} \label{Figure-27c}\vspace{-1mm}
\end{figure}

The triples of numbers near each vertex encode the corresponding Casimir function: red numbers correspond to the smallest Casimir function (Case~\ref{case1}), green numbers correspond to the intermediate Casimir function (Case~\ref{case2}), and blue numbers encode the Casimir function of Case~\ref{case3}. More precisely, the Casimir function is obtained as the product of cluster variables raised to the power indicated by the number printed in the corresponding color.

The last quiver above can be redrawn as the next quiver depicted in Figure~\ref{Figure-27d}.

\begin{figure}[!ht]
\centering
\includegraphics[width=45mm]{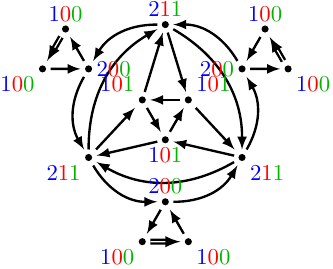}
\vspace{-1mm}

\caption{This quiver is equivalent to the quiver in Figure~\ref{Figure-27c}.} \label{Figure-27d}\vspace{-1mm}
\end{figure}

Now consider the following sequence of mutations performed on one of the ``wings'' of this quiver (see Figure~\ref{Figure-27e}). Since we only mutate the bottom three vertices of the quiver, we need consider only those vertices and their neighbors (the highlighted region).
Applying these mutations to all three wings of the quiver yields the quiver depicted in Figure~\ref{Figure-27f}.

\begin{figure}[!ht]
\centering
\includegraphics[width=95mm]{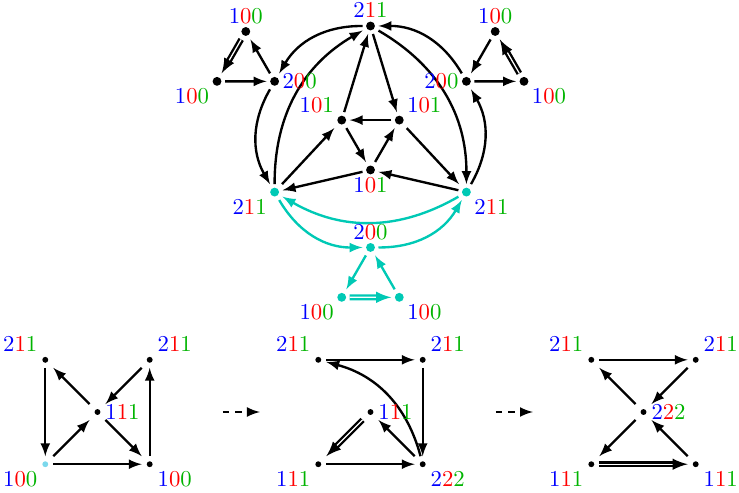}

\caption{Three consecutive mutations at the vertices of the bottom highlighted triangle in the cyclic counterclockwise order starting from the top vertex.} \label{Figure-27e}
\end{figure}

\begin{figure}[!ht]
\centering
\includegraphics[width=65mm]{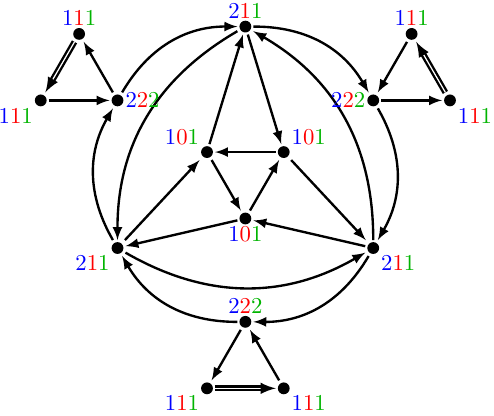}

\caption{Result of the three sequences of mutations from Figure~\ref{Figure-27e} applied to all three triangles (wings) of the quiver in Figure~\ref{Figure-27e}.} \label{Figure-27f}
\end{figure}

We are done with mutations of the wings, and now turn our attention to the highlighted area depicted in Figure~\ref{Figure-27g}.

\begin{figure}[!ht]
\centering
\includegraphics[width=65mm]{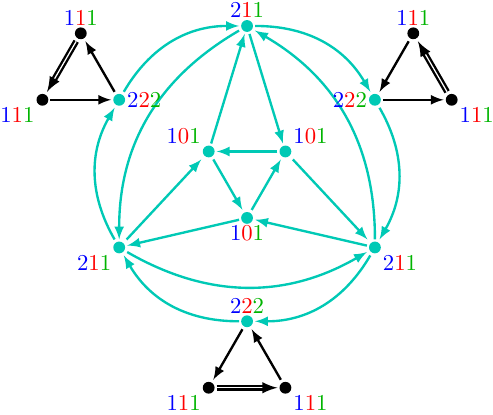}

\caption{Subsequent mutation sequences are located inside the highlighted area.} \label{Figure-27g}
\end{figure}

\begin{figure}[!ht]
\centering
\includegraphics[width=92mm]{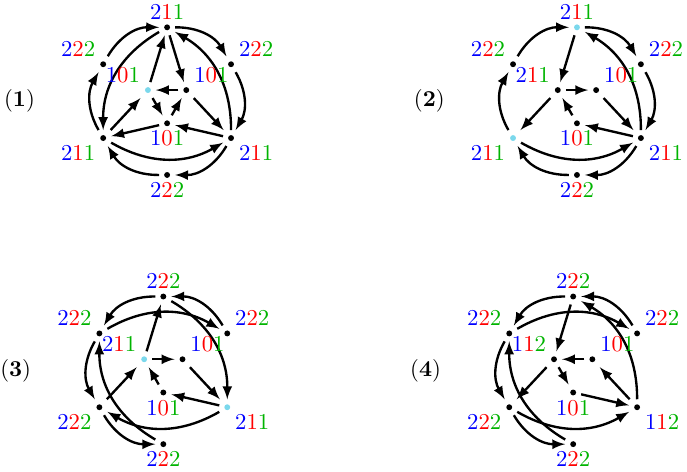}
\vspace{-2mm}

\caption{A mutation sequence that interchanges the first and third Casimirs.} \label{Figure-27h}\vspace{-1mm}
\end{figure}

We now present a sequence of mutations on this quiver that interchanges the first and third Casimirs (see Figure~\ref{Figure-27h}).
Observe that the Casimirs of Cases~\ref{case1} and~\ref{case3} are interchanged in the final two quivers. Thus, relabeling the two most recently mutated vertices and repeating the sequence of mutations that led to this second-to-last quiver in reverse returns us to the original quiver, but with the Casimirs of Cases~\ref{case1} and~\ref{case3} interchanged.

Next, in Figure~\ref{Figure-27i} we present a sequence of vertex swaps that cyclically permutes the Casimirs.

\begin{figure}[!ht]
\centering
\includegraphics[width=92mm]{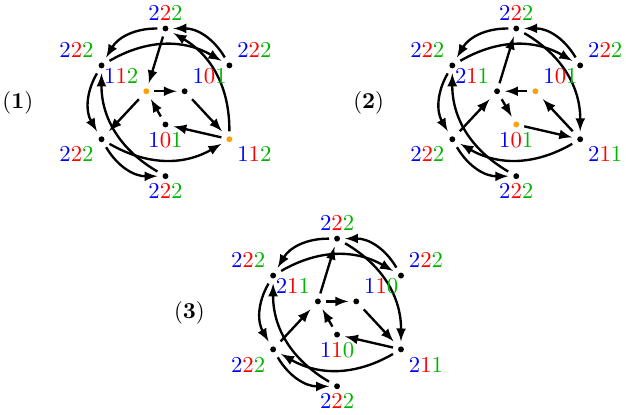}
\vspace{-2mm}

\caption{A mutation sequence that cyclically permutes the Casimirs.} \label{Figure-27i}\vspace{-1mm}
\end{figure}

If we now relabel two pairs of variables, we obtain the starting quiver~(1) with the Casimirs of Cases~\ref{case1},~\ref{case2} and~\ref{case3} permuted cyclically in that order. Repeating now the sequence of mutations in reverse, we obtain the original $\mathcal A_6$-quiver with two pairs of relabeled vertices and with all three Casimirs cyclically permuted.

Straightforward verification using \textsc{Maple} confirms the following statement.

\begin{Lemma}\samepage The mutation sequences constructed above satisfy the following conditions:
\begin{itemize}\itemsep=0pt
 \item The first sequence of mutations permutes the Casimirs of Cases~{\rm \ref{case2}} and~{\rm \ref{case3}}, leaving the Casimir of Case~{\rm \ref{case1}} unchanged. It also permutes the expressions for the minors defining the second conditions of Cases~{\rm \ref{case2}} and~{\rm \ref{case3}}. Therefore, we conclude that the first sequence of mutations permutes Cases~{\rm \ref{case2}} and~{\rm \ref{case3}}, leaving Case~{\rm \ref{case1}} unchanged.
 \item The second mutation sequence cyclically permutes Cases~{\rm \ref{case1}}, {\rm \ref{case2}}, and~{\rm \ref{case3}}.
\end{itemize}
\end{Lemma}

\begin{Remark}
 Since the transposition $(2,3)$ and $3$-cycle $(1,2,3)$ generate the whole group $S_3$ of permutations of three elements, combining these sequences of mutations, one can obtain any permutation of Cases~\ref{case1},~\ref{case2}, and~\ref{case3}.
\end{Remark}

\begin{Remark}
As in the case $n=5$, the coisotropic reduction for $n=6$ leads to the cluster algebra con\-ta\-ining all exponential shear coordinates as cluster variables and, therefore, all geodesic functions (symmetric or not) are universal Laurent polynomials of the square roots of these cluster variables.
\end{Remark}

\begin{Remark}\label{rem:Berstein}
In the paper~\cite{GONCHAROV2018225}, three sequences of mutations were constructed that act on the set of Casimirs.
Let us introduce the following notation: $A=a_1 a_2 \dots a_6$, $B=b_1 b_2 \dots b_6$, $C=c_1 c_2 c_3$; $\hat A= A B C$, $\hat B= B C$, $\hat C=C$.
The equalities $\hat A=\hat B=\hat C=1$ correspond to the vanishing of anti-diagonal entries of matrix~$\Psi$.

These sequences of mutations lead to the following rational transformations:
\begin{itemize}\itemsep=0pt
 \item $s_a\colon (A,B,C)\mapsto \bigl(A^{-1},AB,C\bigr)$;
 \item $s_b\colon (A,B,C)\mapsto \bigl(AB,B^{-1},BC\bigr)$;
 \item $s_c\colon (A,B,C)\mapsto \bigl(A,BC,C^{-1}\bigr)$.
\end{itemize}
Equivalently,
\begin{itemize}\itemsep=0pt
 \item $s_a\colon \bigl(\hat A,\hat B,\hat C\bigr)\mapsto \bigl(\hat B,\hat A,\hat C\bigr)$;
 \item $s_b\colon (A,B,C)\mapsto \bigl(\hat A,\hat C,\hat B\bigr)$;
 \item $s_c\colon (A,B,C)\mapsto \bigl(\hat A\hat C^{-1},\hat B\hat C^{-1},\hat C^{-1}\bigr)$.
\end{itemize}
Note that the subgroup generated by $s_a$, $s_b$ acts as the permutation group $S_3$ on the triple of Casimirs $\hat A$, $\hat B$, $\hat C$.
All three generators $s_a$, $s_b$, $s_c$ satisfy the following relations:
$s_a^2=s_b^2=s_c^2=1$, $(s_a s_b)^3=1$, $(s_a s_c)^2=1$, $(s_b s_c)^4=1$, i.e.,
$s_a$, $s_b$, $s_c$ generate the Weyl group of type $B_3\simeq C_3$. It would be interesting to understand the connection between this Weyl group and the canonical braid group action on~$\Acc_n$. In~\cite{GONCHAROV2018225}, the Weyl group acts on the moduli space of framed local systems on oriented surfaces by permutation of the eigenvalues and the change of flags of eigenvectors of monodromy operators. It might be interesting to find a similar geometric description of the Weyl group action on the symplectic groupoid.
We are indebted to M.~Berstein for this observation.
Note, however, that mutation values $c_2=-1$ and $c_3=-1$ are not in
the set of allowed values
for mutations $\mu_{c_2}$ and $\mu_{c_3}$ so $s_c$ is not allowed in the procedure of geometric reduction.\looseness=-1
\end{Remark}
\begin{Remark}
 A.~Shapiro and G.~Schrader constructed an embedding of $U_q({\rm GL}_n)$ into the quantum cluster algebra of framed local ${\rm GL}_n$ systems on a punctured torus with two marked points on the boundary. The latter has a natural Weyl group action, and elements of~$U_q({\rm GL}_n)$ are invariant with respect to this Weyl group action. L.~Shen~\cite{shen2022clusternaturequantumgroups} proved that this embedding is, in fact, an isomorphism. Similarly, we assume that the algebra of regular functions on~$\Acc_n$ is isomorphic to the subalgebra of elements of the cluster algebra invariant under the corresponding Weyl group action.
\end{Remark}

\appendix

\section[Positivity condition for n=5]{Positivity condition for $\boldsymbol{n=5}$}\label{sec:positivityA5}

In this and the next appendices, we further investigate the coisotropic reduction to the geometric locus, taking into account the positivity of the cluster variables (exponentiated shear coordinates), which means that $h_i>0$ and $v_j>0$ in the case $n=5$. We solve this condition in terms of $a_i$, $b_j$ with the coisotropic reduction condition imposed. Our main goal is to explore how the corresponding set of inequalities is transformed under the braid group action.

\subsection[Inequalities for A\_5 cluster variables]{Inequalities for $\boldsymbol{\mathcal A_5}$ cluster variables}

\subsubsection{The base system}

Let
$Q(x,y)=x^2y^2+2x^2y+x^2-2xy+2x+1$.
The {\it base system} $(A)$ that expresses the positivity condition for geometric cluster variables $h_i$, $i=1,\dots,4$, and $v_j$, $j=1,\dots,5$, reads{\samepage
\begin{align*}
&a_5=-\frac1{a_1a_2},\qquad
 a_1a_2a_3^2>0,\qquad
 a_1a_2a_3+a_1a_3+a_3>0, \\
&a_5a_1a_4^2>0,\qquad
 a_5a_1a_4+a_5a_4+a_4>0,\qquad
 Q(a_1,a_2)>0, \\
&b_1,b_2,b_3,b_4,b_5>0.
\end{align*}
Note that, for $a_5a_1a_2=-1$, $Q(a_5,a_1)=\frac{1}{a_1^2a_2^2}Q(a_1,a_2)$, so $Q(a_5,a_1)>0$.}

Equivalently, if
\[
S:=1+a_1(a_2+1),\qquad T:=1+a_5(a_1+1),
\]
then the two sign inequalities involving $a_3$ and $a_4$ are
$
a_3S>0$,
$a_4T>0$.

From $a_5=-1/(a_1a_2)$ and $a_5a_1a_4^2>0$, one gets $a_2<0$. Then $a_1a_2a_3^2>0$ gives $a_1<0$. Write
$
a_1=-u$,
$a_2=-r^2$,
$ u>0$,
$r>0$.
Then
\[
a_5=-\frac1{ur^2},
\]
and the last quadratic inequality becomes
\[
u^2\bigl(r^2-1\bigr)^2-2u\bigl(r^2+1\bigr)+1>0.
\]
Its roots are
\[
\frac1{(r+1)^2},
\qquad
\frac1{(r-1)^2}
\]
when $r\ne1$, and for $r=1$ the same conclusion is obtained by taking the limit. Thus, the solution set of (A) has three sign components:
\begin{align*}
{\rm (A.1)}\colon \ & r>1,\qquad u>\frac1{(r-1)^2},\qquad a_3>0,
\qquad a_4>0,\\
{\rm (A.2)}\colon \ & 0<u<\frac1{(r+1)^2},\qquad a_3>0,
\qquad a_4<0,\\
{\rm (A.3)}\colon \ & 0<r<1,\qquad u>\frac1{(1-r)^2},\qquad a_3<0,
\qquad a_4>0.
\end{align*}
For any choice of variables $(x_1,\dots,x_5,y_1,\dots,y_5)$, let $\A(x,y)$ denote the same system as~(A), with $a_i$, $b_i$ replaced by $x_i$, $y_i$:
\begin{align*}
&x_5=-\frac1{x_1x_2},\qquad
 x_1x_2x_3^2>0,
\qquad x_3 (1+x_1(x_2+1) )>0,\\
&x_5x_1x_4^2>0,
\qquad x_4 (1+x_5(x_1+1) )>0,
\qquad Q(x_1,x_2)>0,\\
&y_1,y_2,y_3,y_4,y_5>0.
\end{align*}

\subsubsection{Inverse formulas}

Assume that the system (A) holds. Put
\[
D:=Q(a_1,a_2)
=a_1^2a_2^2+2a_1^2a_2+a_1^2-2a_1a_2+2a_1+1.
\]
The geometric cluster variables are then expressed via the $\mathcal A_5$ cluster variables (with the condition $a_1a_2a_5=-1$ imposed):
\begin{alignat*}{3}
&h_1=\frac{a_3}{2}\bigl(a_1a_2+a_1+1-\sqrt D\bigr),\qquad &&
 h_4=\frac{a_3}{2}\bigl(a_1a_2+a_1+1+\sqrt D\bigr),& \\
&h_2=\frac{a_4}{2a_1a_2}\bigl(a_1a_2-a_1-1-\sqrt D\bigr),\qquad &&
 h_3=\frac{a_4}{2a_1a_2}\bigl(a_1a_2-a_1-1+\sqrt D\bigr).&
\end{alignat*}
A remarkable fact is that the above set of relations can be \emph{inverted}, producing {\it rational} expressions for $a_i$ in terms of the geometric cluster variables $h_i$, $i=1,\dots,4$:
\begin{alignat}{3}
&a_1=-\frac{h_2h_3(h_1-h_4)^2}{h_1h_4(h_2-h_3)^2},\qquad &&
 a_2=-\frac{(h_1h_2-h_3h_4)^2}{h_2h_3(h_1-h_4)^2},& \nonumber\\
&a_3=\frac{h_1h_4(h_2-h_3)}{h_1h_2-h_3h_4},\qquad &&
 a_4=\frac{h_1h_2-h_3h_4}{h_1-h_4},&\nonumber\\
&a_5=-\frac{h_1h_4(h_2-h_3)^2}{(h_1h_2-h_3h_4)^2}.\qquad &&&\label{eq:AX}
\end{alignat}
First, consider the Poisson transformation of geometric variables $h_j\to h_j^{-1}$ for $j=1,\dots,4$. The corresponding transformation for $a_i$,
\begin{equation}\label{eq:a-inverse}
a_i\to a_i, \qquad i=1,2,5,\qquad a_3\to (h_1h_4)^{-1} a_3,\qquad a_4\to (h_2h_3)^{-1} a_4,
\end{equation}
is a Poisson morphism.

All geometric cluster variables are independent and positive, $h_j>0$. We then consider a~subset $\mathcal R$ of $h_j$
\[
\mathcal R:=\mathbb R^4_+\backslash \{h_1=h_4 \lor h_2=h_3\lor h_1h_2=h_3h_4\}
\]
with singularities at $h_1=h_4$, $h_2=h_3$, and $h_1h_2=h_3h_4$ removed. The domain $\mathcal R$ is then split into six nonintersecting domains in~$\mathbb R_+^4$:
\begin{alignat*}{2}
&\hbox{(I)}\quad &&h_1>h_4\ \hbox{and}\ h_2>h_3\quad \hbox{(then $h_1h_2>h_3h_4$ and $a_3>0$ and $a_4>0$)},\\
&\hbox{(II)}\quad &&h_1<h_4\ \hbox{and}\ h_2<h_3\quad \hbox{(then $h_1h_2<h_3h_4$ and $a_3>0$ and $a_4>0$)},\\
&\hbox{(III)}\quad &&h_1>h_4\ \hbox{and}\ h_2<h_3\ \hbox{and}\ h_1h_2<h_3h_4\quad \hbox{(then $a_3>0$ and $a_4<0$)},\\
&\hbox{(IV)}\quad &&h_1<h_4\ \hbox{and}\ h_2>h_3\ \hbox{and}\ h_1h_2>h_3h_4\quad \hbox{(then $a_3>0$ and $a_4<0$)},\\
&\hbox{(V)}\quad &&h_1>h_4\ \hbox{and}\ h_2<h_3\ \hbox{and}\ h_1h_2>h_3h_4\quad \hbox{(then $a_3<0$ and $a_4>0$)},\\
&\hbox{(VI)}\quad &&h_1<h_4\ \hbox{and}\ h_2>h_3\ \hbox{and}\ h_1h_2<h_3h_4\quad \hbox{(then $a_3<0$ and $a_4>0$)}.
\end{alignat*}
The sign components in the solution of system (A) therefore correspond to the respective sectors: (A.1) to (I)$\cup$(II), (A.2) to (III)$\cup$(IV), and (A.3) to (V)$\cup$(VI), and the map $\{h_j\}\in\mathcal R\to \{a_i\}$ is therefore a double cover with variables in the dual sectors related by the transformation (\ref{eq:a-inverse}).\footnote{We leave the problem of exploring boundary behavior of the reduction mappings to future work.}

Next we consider the action of braid-group transformations (maps) $F$, $F_1$, $F_2$ on system~(A). These maps produce the corresponding systems (B), (C), (D). In all calculations below we always assume that all denominators do not vanish. The braid-group transformations on all quivers~$\mathcal A_n$ are given by particular sequences of mutations derived in~\cite{10.1093/imrn/rnac101}. In the original $\mathcal A_5$-quiver in Figure~\ref{quiverA5}, the four generators are
$\Omega_i:=\mu_{a_i}\mu_{a_{i+2}}\mu_{b_{i+1}}\mu_{a_{i+2}}\mu_{a_i}$
for $i=1,\dots,4$.

\subsection[Case 1: The map F and system (B)]{Case 1: The map $\boldsymbol{F}$ and system (B)}

The map $F$ corresponds to the braid-group transformation $\Omega_1$. Define
\[
X:=a_1a_2a_3^2,
\qquad
S:=1+a_1(a_2+1).
\]
On (A), one has
$X>0$,
$ a_3S>0$.
Put
\begin{alignat*}{2}
&E:=1+b_2(1+a_3),\qquad
&&C:=1+b_2\bigl(1+a_3(a_1+1)\bigr),\\
&H:=1+b_2(1+X+a_3S),\qquad
&&R:=1+(1+X+a_3S)b_2^2+(2+a_3S)b_2.
\end{alignat*}
Then
$H>0$,
$R>0$.
The map $F$ can be equivalently written as
\begin{alignat*}{4}
&u_1=\frac{a_1R}{C^2},\qquad &&
 u_2=\frac{a_2E^2}{R},\qquad &&
 u_3=\frac{a_3C}{HE},&\\
& u_4=a_4E,\qquad &&
 u_5=\frac{a_5C^2}{E^2},\qquad &&
 v_1=\frac{b_1H^2}{R},&\\
& v_2=\frac{R}{a_1a_2a_3^2b_2},\qquad &&
 v_3=\frac{a_1a_2a_3^2b_3b_2^2}{R},\qquad &&
v_4=b_4,
\qquad v_5=b_5.&
\end{alignat*}
The domain of this factored formula is
$C\ne0$,
$E\ne0$.

\subsubsection*{Derivation of system (B)}

The first equality in (B) is just
\[
u_5=-\frac1{u_1u_2},
\]
because
\[
u_1u_2=\frac{a_1a_2E^2}{C^2}.
\]
Moreover,
\begin{align*}
&u_1u_2u_3^2
 =\frac{a_1a_2a_3^2}{H^2}>0,\qquad
 u_3 (1+u_1(u_2+1) )
 =\frac{a_3S}{H}>0,\qquad
 u_5u_1u_4^2
 =\bigl(a_5a_1a_4^2\bigr)R>0.
\end{align*}
For the fourth sign inequality, write
$a_1=-u$,
$a_2=-r^2$.
Then
\[
u_4 (1+u_5(u_1+1))
=-\frac{a_4J}{a_1a_2},
\]
where
\[
J=-\bigl(u\bigl(r^2+1\bigr)-1\bigr)
+b_2\bigl[-\bigl(u\bigl(r^2+1\bigr)-1\bigr)
+a_3\bigl(\bigl(1-r^2\bigr)u^2-\bigl(r^2+2\bigr)u+1\bigr)\bigr].
\]
In (A.1) and (A.3) one has $a_4>0$ and $J<0$; in (A.2) one has $a_4<0$ and $J>0$. Hence
\[
u_4 (1+u_5(u_1+1) )>0.
\]
Finally,
\[
Q(u_1,u_2)=\frac{E^2}{C^2}Q(a_1,a_2)>0.
\]
The inequalities $v_i>0$ follow from $b_i>0$, $H>0$, $R>0$, and $a_1a_2a_3^2>0$. Thus (A) implies (B) wherever $F$ is defined.

\subsubsection*{Case correspondence for $\boldsymbol{F}$}

Since
\[
\sgn(u_3)=\sgn(a_3)\sgn(C)\sgn(E),
\qquad
\sgn(u_4)=\sgn(a_4)\sgn(E),
\]
the correspondence is
$$
\boxed{
\begin{array}{@{}c|c|c@{}}
\text{input component} & \text{extra condition} & \text{output component} \\
\hline
{\rm (A.1)} & C>0 & {\rm (B.1)}\\
{\rm (A.1)} & C<0 & {\rm (B.3)}\\
{\rm (A.2)} & \text{always} & {\rm (B.2)}\\
{\rm (A.3)} & E>0 & {\rm (B.3)}\\
{\rm (A.3)} & E<0 & {\rm (B.2)}
\end{array}}
$$
where
$C=1+b_2(1+a_3(a_1+1))$,
$E=1+b_2(1+a_3)$.
Thus
\[
{\rm (A.1)}\mapsto {\rm (B.1)}\cup{\rm (B.3)},
\qquad
{\rm (A.2)}\mapsto {\rm (B.2)},
\qquad
{\rm (A.3)} \mapsto {\rm (B.2)}\cup{\rm (B.3)}.
\]

\subsection[Case 2: The map F\_1 and system (C)]{Case 2: The map $\boldsymbol{F_1}$ and system (C)}

The map $F_1$ corresponds to the braid-group transformation $\Omega_4$.
Keep
\[
S:=1+a_1(a_2+1),
\qquad
T:=1+a_5(a_1+1).
\]
Also put
$Y:=a_1a_5a_4^2$.
On (A),
$Y>0$,
$a_4T>0$.
Define
\begin{alignat*}{2}
&E_1:=1+b_3(1+a_4),\qquad
&&G_1:=1+b_3 (1+a_4(1+a_5) ),\\
&H_1:=1+b_3(1+Y+a_4T),\qquad
&&R_1:=1+(1+Y+a_4T)b_3^2+(2+a_4T)b_3.
\end{alignat*}
Then
$H_1>0$,
$R_1>0$.
Using $a_5=-1/(a_1a_2)$, the map $F_1$ becomes
\begin{alignat*}{4}
&u_1=\frac{a_1E_1^2}{R_1},\qquad &&
 u_2=\frac{a_2G_1^2}{E_1^2},\qquad &&
 u_3=\frac{a_3H_1}{G_1},&\\
& u_4=\frac{a_4G_1}{E_1H_1},\qquad &&
 u_5=\frac{a_5R_1}{G_1^2},\qquad &&
 v_1=b_1,&\\
& v_2=b_2,\qquad &&
 v_3=-\frac{a_2R_1}{a_4^2b_3},\qquad &&
 v_4=-\frac{a_4^2b_4b_3^2}{a_2R_1},
\qquad
 v_5=b_5R_1.&
\end{alignat*}
The domain is
$
E_1\ne0$,
$G_1\ne0$.

\subsubsection*{Derivation of system (C)}

First,
\[
u_5=-\frac1{u_1u_2}.
\]
The basic sign inequalities factor as
\begin{align*}
&u_1u_2u_3^2
 =\frac{a_1a_2a_3^2H_1^2}{R_1}>0,\qquad
 u_5u_1u_4^2
 =\frac{a_5a_1a_4^2}{H_1^2}>0,\qquad
 Q(u_1,u_2)
 =\frac{E_1^2G_1^2}{R_1^2}Q(a_1,a_2)>0.
\end{align*}
For the remaining two inequalities, define
\[
L_1:=S+b_3 (S(1+a_4)-2a_4 ),
\qquad
W:=a_1a_2-a_1-1.
\]
Then
\begin{align*}
&u_3 (1+u_1(u_2+1) )
 =\frac{a_3H_1L_1}{R_1},\qquad
 u_4 (1+u_5(u_1+1) )
 =\frac{a_4W}{a_1a_2H_1}.
\end{align*}
The factor $L_1$ has the same sign as $S$. Indeed, in (A.1) one has $S>2$ and $a_4>0$, so
\[
S(1+a_4)-2a_4=S+a_4(S-2)>0.
\]
In (A.2), one has $0<S<2$ and $a_4<0$, so the same expression is again positive. In (A.3), one has $S<0$ and $a_4>0$, so it is negative. Hence $a_3L_1>0$. Also
$W=a_1a_2T$,
and since $a_1a_2>0$ and $a_4T>0$, one has $a_4W>0$. Therefore, the two remaining inequalities hold.

The inequalities $v_i>0$ follow from $b_i>0$, $a_2<0$, and $R_1>0$. Thus (A) implies (C) wherever $F_1$ is defined.

\subsubsection*{Case correspondence for $\boldsymbol{F_1}$}

Since
\[
\sgn(u_3)=\sgn(a_3)\sgn(G_1),
\qquad
\sgn(u_4)=\sgn(a_4)\sgn(G_1)\sgn(E_1),
\]
the correspondence is
$$
\boxed{
\begin{array}{@{}c|c|c@{}}
\text{input component} & \text{extra condition} & \text{output component} \\
\hline
{\rm(A.1)} & \text{always} & {\rm(C.1)}\\
{\rm(A.2)} & E_1>0 & {\rm(C.2)}\\
{\rm(A.2)} & E_1<0 & {\rm(C.1)}\\
{\rm(A.3)} & G_1>0 & {\rm(C.3)}\\
{\rm(A.3)} & G_1<0 & {\rm(C.2)}
\end{array}}
$$
where
$E_1=1+b_3(1+a_4)$,
$G_1=1+b_3(1+a_4(1+a_5))$.
Thus
\[
{\rm(A.1)}\mapsto {\rm(C.1)},
\qquad
{\rm(A.2)}\mapsto {\rm(C.1)}\cup{\rm(C.2)},
\qquad
{\rm (A.3)}\mapsto {\rm(C.2)}\cup{\rm(C.3)}.
\]

\subsection[Case 3: The map F\_2 and system (D)]{Case 3: The map $\boldsymbol{F_2}$ and system (D)}

The map $F_2$ corresponds to the braid-group transformation $\Omega_2$.
Set $b=b_5$. Define
\begin{alignat*}{3}
&L_2:=1+b\ell_2,
\qquad &&
\ell_2:=1+a_3 (1+a_1(a_4+1) ),&\\
&M_2:=1+bm_2,
\qquad &&
m_2:=1+a_1a_2(a_4+1)a_3^2
 + (1+(a_4+1)(a_2+1)a_1 )a_3.&
\end{alignat*}
Next define
\begin{align*}
& Y_2:= 1+ (a_4a_5(a_4+1)a_1+1+a_4(a_5+1) )a_2a_1a_3^2\\
& \hphantom{Y_2:=}{}
+ (1+ ((1+a_4(a_5+1))a_2+a_4+1 )a_1 )a_3,
\end{align*}
and
\[
N_2:=1+bY_2.
\]
Finally, put
\begin{align*}
&p_2:=2+
 (1+ ((1+a_4(a_5+1))a_2+a_4+1 )a_1 )a_3,\\
&d_2:=2+(1+a_4(a_5+1))a_1a_2a_3^2
+ (2+ ((1+a_4(a_5+1))a_2+a_4+1 )a_1 )a_3,\\
&\theta_2:=2+a_2 (a_1a_5a_4+2+a_4(a_5+1) )a_1a_3^2
+ (2+ ((2+a_4(a_5+1))a_2+a_4+2 )a_1 )a_3,
\end{align*}
and
\begin{align*}
&P_2:=1+Y_2b^2+p_2b,\\
&\Delta_2:=1+(1+a_3)Y_2b^2+d_2b,\\
&\Theta_2:=1+\bigl(1+a_1a_2a_3^2+a_3S\bigr)Y_2b^2+\theta_2b.
\end{align*}

With these factors, the map $F_2$ becomes
\begin{alignat*}{4}
&u_1=\frac{a_1P_2M_2^2}{L_2^2\Theta_2},\qquad &&
 u_2=\frac{a_2\Delta_2^2}{P_2M_2^2},\qquad &&
 u_3=\frac{a_3L_2}{\Delta_2},&\\
& u_4=\frac{a_4\Delta_2}{M_2N_2},\qquad &&
 u_5=-\frac{L_2^2\Theta_2}{a_1a_2\Delta_2^2},\qquad &&
 v_1=-\frac{a_1a_3^2a_4^2b_1b^2}{P_2},&\\
& v_2=b_2P_2,\qquad &&
 v_3=\frac{b_3\Theta_2}{P_2},\qquad &&
 v_4=\frac{b_4N_2^2}{\Theta_2},
\qquad
 v_5=-\frac{\Theta_2}{a_1a_3^2a_4^2b}.&
\end{alignat*}
The domain reads
$
L_2\ne0$,
$M_2\ne0$,
$\Delta_2\ne0$.
The other potential denominators are automatically positive on (A), as explained below.

\subsubsection*{Sign facts for the factors}

On (A),
$Y_2>0$,
$N_2>0$,
$P_2>0$,
$\Theta_2>0$.
Furthermore,
$$
\boxed{
\begin{array}{@{}c|ccc@{}}
\text{input component} & L_2 & M_2 & \Delta_2 \\
\hline
{\rm (A.1)} & \pm & + & +\\
{\rm (A.2)} & + & \pm & +\\
{\rm (A.3)} & + & + & \pm
\end{array}}
$$
The signs not fixed by (A) are exactly the factors that distinguish the image component.

The verification is elementary after the substitution $a_1=-u$, $a_2=-r^2$. For example, in~(A.1), one has
\[
r>1,
\qquad
u>\frac1{(r-1)^2},
\]
so $u\bigl(r^2-1\bigr)>1$. This makes the coefficients of
$M_2$, $Y_2$, $p_2$, $d_2$, and $\theta_2$ positive. Hence
\[
M_2>0,\qquad P_2>0,\qquad \Delta_2>0,\qquad \Theta_2>0,\qquad N_2>0.
\]
In (A.2), the inequality $u<1/(r+1)^2$ gives the corresponding
positivity of $L_2$, $Y_2$, $p_2$, $d_2$, $\theta_2$. In~(A.3), the inequalities
$0<r<1$ and $u>1/(1-r)^2$ imply $u\bigl(1-r^2\bigr)>1$. This gives positivity
of $L_2$, $M_2$, $Y_2$, $p_2$, $\theta_2$, while $\Delta_2$ may change sign.

\subsubsection*{Derivation of system (D)}

The first equality in (D) is again
\[
u_5=-\frac1{u_1u_2}.
\]
The basic products reduce to
\begin{align*}
\begin{split}
&u_1u_2u_3^2
=\frac{a_1a_2a_3^2}{\Theta_2}>0,\qquad
 u_5u_1u_4^2
 =-\frac{P_2a_4^2}{a_2N_2^2}>0,\\
& Q(u_1,u_2)
 =\frac{M_2^2\Delta_2^2}{L_2^2\Theta_2^2}Q(a_1,a_2)>0.
 \end{split}
\end{align*}
The positivity of the $v_i$'s follows immediately from the displayed formula for $F_2$, because
$a_1<0$,
$P_2>0$,
$\Theta_2>0$,
$N_2>0$,
$b_i>0$.

It remains to record the two transformed mixed inequalities. Define
$\Gamma_2:=S+b\gamma_2$,
where
\begin{align*}
&K_2:=\bigl((1+a_4(a_5+1))a_2^2+(2a_4+2)a_2+a_4+1\bigr),\\
&\gamma_2= a_2(1+((1+a_4(a_5+1))a_2+a_4+1)a_1)a_1a_3^2\\
&\hphantom{\gamma_2=}{} +\bigl(1+K_2a_1^2+(2a_2+a_4+2)a_1\bigr)a_3+S.
\end{align*}
Also define
\begin{align*}
&\Lambda_2:= a_1^2a_2a_3a_4b+a_1^2a_2a_3b+a_1^2a_3a_4b+a_1^2a_3b
-a_1a_2a_3b-a_1a_2b-a_1a_2\\
&\hphantom{\Lambda_2:=}{} +a_1a_3a_4b+2a_1a_3b+a_1b+a_1+a_3b+b+1.
\end{align*}
A direct factorization gives
\begin{align*}
&u_3 (1+u_1(u_2+1) )
=\frac{a_3\Gamma_2}{\Theta_2},\qquad
 u_4 (1+u_5(u_1+1) )
 =-\frac{a_4\Lambda_2}{a_1a_2N_2}.
\end{align*}
Under (A),
$a_3\Gamma_2>0$,
$-a_4\Lambda_2>0$.
Again this follows by substituting $a_1=-u$, $a_2=-r^2$ and using the three parameter ranges above. Therefore, the two remaining inequalities in (D) hold.

Thus (A) implies (D) wherever $F_2$ is defined.

\subsubsection*{Case correspondence for $\boldsymbol{F_2}$}

Since
\[
\sgn(u_3)=\sgn(a_3)\sgn(L_2)\sgn(\Delta_2),
\qquad
\sgn(u_4)=\sgn(a_4)\sgn(M_2)\sgn(\Delta_2),
\]
the correspondence is
$$
\boxed{
\begin{array}{@{}c|c|c@{}}
\text{input component} & \text{extra condition} & \text{output component} \\
\hline
{\rm (A.1)} & L_2>0 & {\rm (D.1)}\\
{\rm (A.1)} & L_2<0 & {\rm (D.3)}\\
{\rm (A.2)} & M_2>0 & {\rm (D.2)}\\
{\rm (A.2)} & M_2<0 & {\rm (D.1)}\\
{\rm (A.3)} & \Delta_2>0 & {\rm (D.3)}\\
{\rm (A.3)} & \Delta_2<0 & {\rm (D.2)}
\end{array}}
$$
Thus
\[
{\rm (A.1)}\mapsto {\rm (D.1)}\cup{\rm (D.3)},
\qquad
{\rm (A.2)}\mapsto {\rm (D.1)}\cup{\rm (D.2)},
\qquad
{\rm (A.3)}\mapsto {\rm (D.2)}\cup{\rm (D.3)}.
\]

\subsection{Summary of the three correspondences}

We summarize our findings for $n=5$ in the following table:
$$
\begin{array}{c|c|c}
\toprule
\text{map} & \text{input component of }(A) & \text{output components}\\
\midrule
F &{\rm (A.1)}&{\rm (B.1)}\cup{\rm (B.3)}\\
F &{\rm (A.2)}&{\rm (B.2)}\\
F &{\rm (A.3)}&{\rm (B.2)}\cup{\rm (B.3)}\\
\midrule
F_1 &{\rm (A.1)}&{\rm (C.1)}\\
F_1 &{\rm (A.2)}&{\rm (C.1)}\cup{\rm (C.2)}\\
F_1 &{\rm (A.3)}&{\rm (C.2)}\cup{\rm (C.3)}\\
\midrule
F_2 &{\rm (A.1)}&{\rm (D.1)}\cup{\rm (D.3)}\\
F_2 &{\rm (A.2)}&{\rm (D.1)}\cup{\rm (D.2)}\\
F_2 &{\rm (A.3)}&{\rm (D.2)}\cup{\rm (D.3)}\\
\bottomrule
\end{array}
$$
All boundary cases where the splitting factor is zero are excluded from the corresponding rational map domain.

\section[Positivity condition for n=6]{Positivity condition for $\boldsymbol{n=6}$}

We now turn to the case of $n=6$; the analysis for even $n$ seems to be easier since all transformations that occur are birational. To ensure positivity of all variables in the geometric quiver obtained by mutation and amalgamation, we first impose the coisotropic reduction condition
\[
c_1c_2c_3=1,\qquad 1+c_2+c_2c_1=0,\qquad \hbox{or}\qquad c_1=-1-c_2^{-1},\qquad c_3=-(1+c_2)^{-1}.
\]
Then the condition that all cluster variables in the geometric quiver are positive can be split into the three sets of inequalities collected in the table below
$$
\begin{array}{r|c|c|c|}
&\hbox{Case 1}&\hbox{Case 2}&\hbox{Case 3}\\
\hline
c_2&<-1&(-1,0)&>0\\
c_1&(-1,0)&>0&<-1\\
c_3&>0&<-1&(-1,0)\\
\hline
b_1&>0&<0&>0\\
b_2&>0&>0&<0\\
b_3&<0&>0&>0\\
b_4&>0&<0&>0\\
b_5&>0&>0&<0\\
b_6&<0&>0&>0\\
\hline
\end{array}
$$
All variables $a_i$ must be positive in all three cases.

We now perform the Dehn twist induced by a sequence of mutations
$
\mu_{b_1}\mu_{c_2}\mu_{b_3}\mu_{a_3}\mu_{b_3}\mu_{c_2}\mu_{b_1}$
in the $\mathcal A_6$ quiver. Retaining the labels of cluster variables, we have the following transformation~law:
\begin{alignat*}{4}
&c_1'=\frac{c_1c_2\varkappa}{c_2\varkappa+1},\qquad &&c_2'=\frac{c_2\varkappa+1}{\varkappa-1},\qquad &&c_3'=\frac{c_3(\varkappa-1)}{\varkappa}, &\\
&b_1'=\frac{\varkappa}{a_3},\qquad &&b_2'=\frac{b_2a_3\eta}{(1+a_3)(\eta+1)},\qquad &&b_3'=\frac{\eta+1}{c_2\varkappa},& \nonumber\\
&b_4'=\frac{b_4c_2\varkappa}{c_2\varkappa+1},\qquad &&b_5'=b_5,\qquad &&b_6'=\frac{b_6(\varkappa-1)}{\varkappa}, &\nonumber\\
&a_1'=\frac{a_1a_3}{1+a_3},\qquad &&a_2'=a_2(1+a_3)(\eta+1),\qquad &&a_3'=\frac{1}{\eta},&\nonumber\\
&a_4'=\frac{a_4\eta}{\eta+1},\qquad && a_5'=a_5,\qquad &&a_6'=a_6,&
\end{alignat*}
where
\begin{equation}\label{eq:eta-kappa}
\varkappa:=1+b_1(1+a_3),\qquad \eta:=b_3+b_3c_2+b_3c_2b_1+b_3c_2b_1a_3.
\end{equation}

Observe, first, that
\begin{align*}
&c_1'c_2'c_3'=c_1c_2c_3\qquad \hbox{and}\qquad 1+c_2'+c_2'c_1'=\frac{\varkappa}{\varkappa-1}(1+c_2+c_2c_1),
\end{align*}
so the reduction conditions are preserved by the Dehn twist.

We next claim that $\eta>0$ in all three cases. Indeed, when $c_2>0$, all variables in the corresponding expression in (\ref{eq:eta-kappa}) are positive, so $\eta>0$. When $c_2<-1$, $b_3<0$ and $b_1>0$, so $b_1c_2b_3(1+a_3)>0$ and $b_3(1+c_2)>0$ and, therefore, $\eta>0$; when finally $-1<c_2<0$, $b_3>0$ and $b_1<0$ and then, again, $b_1c_2b_3(1+a_3)>0$ and $b_3(1+c_2)>0$, so $\eta>0$. This immediately implies that all variables $a_i'>0$ in all three cases. For the signs of $b_i'$, we then have
\begin{alignat*}{4}
&\sgn b_1'=\sgn \varkappa,\qquad &&\sgn b_2'=\sgn b_2,\qquad &&\sgn b_3'=\sgn (c_2\varkappa),&\\
&\sgn b_4'=\sgn (b_4c_2\varkappa)\sgn (c_2\varkappa+1),\qquad && \sgn b_5'=\sgn b_5,\qquad &&\sgn b_6'=\sgn (b_6(\varkappa-1)\varkappa).&
\end{alignat*}

\subsection{Case 3}
The simplest case is Case~3 in which $c_2>0$. Then $b_1>0$ and $\varkappa>1$, so $c_2'=\frac{c_2\varkappa+1}{\varkappa-1}>0$, $b_2'$~and~$b_5'$ are negative, all other $b_i'$ are positive, and we obtain again Case~3 for the transformed variables.

\subsection{Case 1}
Let $c_2<-1$. Then $b_1>0$, $\varkappa>1$ and $c_2'+1=(c_2+1)\frac{\varkappa}{\varkappa-1}<0$, so $c_2'<-1$. In this case, $c_2\varkappa+1<0$, and it is easy to see that all $b_i'$ are positive except $b_3'$ and $b_6'$, so we again obtain Case~1.

\subsection{Case 2}
The only nontrivial situation occurs in Case~2: $-1<c_2<0$. In this case, $b_1<0$ and $\varkappa<1$. We consider separately two subcases:
\subsubsection[Subcase 0<varkappa<1]{Subcase $\boldsymbol{0<\varkappa<1}$} In this case, $c_2\varkappa+1>0$, and we have that
\[
c_2'+1=(c_2+1)\frac{\varkappa}{\varkappa-1}<0.
\]
The variables $b_2'$, $b_5'$, $b_1'$, and $b_4'$ are positive, whereas $b_3'$ and $b_6'$ are negative. Thus we obtain Case~1.

\subsubsection[Subcase varkappa<0]{Subcase $\boldsymbol{\varkappa<0}$}
In this case, again, $c_2\varkappa+1>0$, and $c_2'+1>0$ but $c_2'<0$. We also have that $b_1'<0$ and $b_4'<0$ while all other $b_i'$ are positive, so we obtain Case~2.

We summarize the results of this section in the following table:
$$
\begin{array}{c|c|c|c|c}
\hbox{input} & \hbox{Case 1}&\hbox{Case 2},\ \varkappa>0 &\hbox{Case 2},\ \varkappa<0&\hbox{Case 3}\\
\hline
\hbox{output} & \hbox{Case 1}&\hbox{Case 1} &\hbox{Case 2}&\hbox{Case 3}
 \end{array}
$$

\begin{Remark}
Note that unlike the case $n=5$, in the $n=6$ case, as well as in all higher-rank cases, we cannot express~$\mathcal A_n$ cluster variables in terms of the geometric cluster variables, and each of the three cases covers the entire Teichm\"uller space parameterized by shear coordinates. This is because the constraint $1+c_2^{-1}+c_2^{-1}c_3^{-1}=0$ with $c_1c_2c_3=1$ is Lagrangian: all composite geometric cluster variables Poisson commute with this constraint on the constraint surface. The product of the $c_i$ is a Casimir of the full algebra and therefore Poisson-commutes with all combinations of variables; by contrast, the second condition results in additional dimensional reduction.
\end{Remark}

\section{Commutativity of reduction mapping and braid-group action}\label{sec:commutativity}

Let $B_{i,i+1}$ and $\widetilde B_{i,i+1}$, $i=1,\dots,n-1$, denote the respective (invertible) braid-group transformations of cluster variables in the $\mathcal A_n$-quiver and in the geometric quivers $\mathcal G_n$. (The sequences of mutations $\Omega_i$ are $B_{i,i+1}$ in the case $n=5$.) We let $\mathcal R$ denote the reduction mapping from $\mathcal A_n$ to $\mathcal G_n$ (this mapping is not invertible for a~general~$n$). We checked by explicit calculations using AI-assisted symbolic computations that the diagram
\[
\begin{tikzcd}
\mathcal A_n \arrow{r}{\mathcal R} \arrow[swap]{d}{B_{i,i+1}} & \mathcal G_n \arrow{d}{\widetilde B_{i,i+1}} \\%
\mathcal A_n \arrow{r}{\mathcal R}& \mathcal G_n
\end{tikzcd}
\]
is commutative for $n=5,6$ in all cases considered above.

\subsection*{Acknowledgements}

The project started under the umbrella of the ``Discovering America'' program at Michigan State University, hosting one-semester joint research teams of
visiting students from China and groups of US students.
The authors are grateful to the Math Department of Michigan State University for its hospitality, support, and inspiring research atmosphere, which enabled us to continue the research that led to this publication. Special thanks are due to Linhui Shen, Francis Bonahon, Alexander Shapiro, and Michael Berstein for the valuable discussions. M.S. was partially supported by the NSF research grant DMS \#2100791.

The computations presented in Appendices~\ref{sec:positivityA5} and~\ref{sec:commutativity} were initially carried out with the assistance of ChatGPT (GPT-5.5). All results were subsequently independently verified by the authors, who take full responsibility for their correctness.

We are especially grateful to the anonymous referees for their careful reading and valuable suggestions, which have significantly improved the quality of this paper.

\pdfbookmark[1]{References}{ref}
\LastPageEnding

\end{document}